\documentclass[11pt,twoside]{amsart}
\usepackage{amssymb}
\usepackage{graphicx,color}

\usepackage[all]{xy}
\usepackage{tikz}
\usetikzlibrary{matrix}
\usepackage{array}
\usepackage[ruled,vlined]{algorithm2e}
\SetAlgoCaptionSeparator{}
\usepackage{fancyvrb}
\usepackage{pdfpages}

\newtheorem{theorem}{Theorem}[section]
\newtheorem{proposition}[theorem]{Proposition}

\newtheorem{definition}[theorem]{Definition}
\newtheorem{notation}[theorem]{Notation}
\newtheorem{example}[theorem]{Example}
\newtheorem{remark}[theorem]{Remark}

\newcommand{\skipit}[1]{{}}
\newcommand{\prfend}{\hbox to7pt{\hfil}
\par\vskip-\baselineskip\hbox to\hsize
{\hfil\vbox {\hrule width6pt height6pt}}\vskip\baselineskip}

\newcommand{\ZZ}{\mathbb{Z}}
\newcommand{\RR}{\mathbb{R}}
\newcommand{\NN}{\mathbb{N}}
\newcommand{\Q}{\mathbb{Q}}

\newcommand {\PP}{\mathbb{P}}

\newcommand{\cM}{\mathcal{M}}

\newcommand{\kk}{\mathbb{K}}

\newcommand{\cA}{\mathcal{A}}
\newcommand{\cH}{\mathcal{H}}
\newcommand{\cV}{X}
\newcommand{\cS}{\mathcal{S}}
\DeclareMathOperator{\reg}{reg}
\newcommand{\bG}{\overline{G}}

\DeclareMathOperator{\GL}{GL}
\DeclareMathOperator{\diag}{diag}
\DeclareMathOperator{\GCD}{GCD}
\DeclareMathOperator{\relint}{relint}
\DeclareMathOperator{\I}{I}

\DeclareMathOperator{\HS}{HS}

\DeclareMathOperator{\HF}{HF}
\DeclareMathOperator{\HP}{HP}
\DeclareMathOperator{\vol}{vol}
\DeclareMathOperator{\codim}{codim}
\DeclareMathOperator{\rk}{rk}
\def\HF{{\operatorname{H\!F}}}
\def\HP{\operatorname{H\!P}}
\def\rank{\operatorname{rank}}
\DeclareMathOperator{\conv}{conv}
\DeclareMathOperator{\rl}{rl}
\DeclareMathOperator{\proj}{Proj}
\newcommand{\myarrow}[2]{\hbox to #1pt{\hfil$\to$\hfil}{\hskip-#1pt{\raise
10pt\hbox to#1pt{\hfil$\scriptscriptstyle #2$\hfil}}}}

\begin{document}

\title[Sumsets and Veronese varieties]
{Sumsets and Veronese varieties}
 \author[L. Colarte-G\'omez]{Liena Colarte-G\'omez} 
 \address{Facultat de
  Matem\`atiques i Inform\`atica, Universitat de Barcelona, Gran Via des les
  Corts Catalanes 585, 08007 Barcelona, Spain} \email{liena.colarte@ub.edu}
\author[J. Elias]{Joan Elias} 
\address{Facultat de
  Matem\`atiques i Inform\`atica, Universitat de Barcelona, Gran Via des les
  Corts Catalanes 585, 08007 Barcelona, Spain} \email{elias@ub.edu}
  \author[R.\ M.\ Mir\'o-Roig]{Rosa M.\ Mir\'o-Roig} 
  \address{Facultat de
  Matem\`atiques i Inform\`atica, Universitat de Barcelona, Gran Via des les
  Corts Catalanes 585, 08007 Barcelona, Spain} \email{miro@ub.edu}

\thanks{The first and third author have been  partially supported by the grant PID2020-113674GB-I00} 
\thanks{The second author has been partially supported by the grant PID2019-104844GB-I00}

\date{\today}

\begin{abstract} In this paper, to any subset $\cA \subset \ZZ^{n}$ we explicitly associate a unique monomial  projection $Y_{n,d_{\cA}}$ of a Veronese variety, whose Hilbert function coincides with the cardinality of the $t$-fold sumsets $t\cA$. This link allows us to tackle the classical problem of determining the polynomial $p_{\cA} \in \Q[t]$ such that $|t\cA| = p_{\cA}(t)$ for all $t \geq t_0$ and the minimum integer $n_0(\cA) \leq t_0$ for which this condition is satisfied, i.e. the so-called {\em phase transition} of $|t\cA|$. 
We use the Castelnuovo--Mumford regularity and the geometry of $Y_{n,d_{\cA}}$ to describe the polynomial $p_{\cA}(t)$ and to derive new bounds for $n_0(\cA)$ under some technical assumptions on the convex hull of $\cA$; and vice versa we apply the theory of sumsets to obtain geometric information of the varieties $Y_{n,d_{\cA}}$. 
\end{abstract}
\maketitle
%%%%%%%%%%%%%%%%%%%%%%%%%%%%%%%%%%%%%%%%%%%%%%%%%
\section{Introduction}
In additive number theory, a {\em $t$-fold sumset} $t\cA \subset \ZZ^{n}$ is the set of sums of $t$ non necessarily different elements of a finite non empty subset $\cA \subset \ZZ^{n}$. A classical problem concerning sumsets consists of determining the behaviour of the cardinality function $\varphi_{\cA}(t) = |t\cA|$ as $t$ grows. A central result of Khovanskii \cite{Khovanskii} shows that there exists a polynomial $p_{\cA}(t) \in \Q[t]$, of degree at most $n$,  such that $\varphi_{\cA}(t) = p_{\cA}(t)$ for $t$ sufficiently large. He also determines the leading coefficient of $p_{\cA}(t)$ in terms of the volume of the convex hull associated to the subset $\cA \subset \ZZ^{n}$ and the index of the additive subgroup in $\ZZ^{n}$ generated by the difference set $\cA - \cA \subset \ZZ^{n}$. Notwithstanding, Khovanskii's result 
sheds no light on the polynomial $p_{\cA}(t)$, excepting the leading coefficient, nor the phase transition of $\varphi_{\cA}(t)$.
In view of these considerations, many contributions have been made to these topics as one can see, for instance, in \cite{Wu-Chen-Chen, Curran-Goldmakher, Elias,  Granville-Walker, Granville-Shakan, Granville-Shakan-Walker, Lee, Nathanson, Nathanson1, Nathanson-Ruzsa}.

Most of the results regarding the polynomial $p_{\cA}(t)$ and the phase transition $n_0(\cA)$ are based on the study of the structure of $t$-fold sumsets $t\cA$ as sets of lattice points, being particularly useful for finite subsets of integers. The case $\cA \subset \ZZ$ can be reduced to finite subsets $\cA = \{0,a_1,\hdots, a_k\} \subset \ZZ$ with $0 < a_1 < \cdots < a_k$ and $\GCD(a_1,\hdots,a_k) = 1$. The structure of the $t-$fold sumset $t\cA$ is completely determined for $t$ sufficiently large and it produces upper bounds for $n_0(\cA)$ (see, for instance,  \cite{Wu-Chen-Chen, Elias, Granville-Walker, Nathanson1}). On the other hand, Khovanskii's result assures that $p_{\cA}(t) = a_kt + b \in \ZZ[t]$. However, when dealing with subsets $\cA \subset \ZZ^{n}$ in arbitrary dimension $n \geq 2$, the structure of $t$-fold sumsets $t\cA$ is considerably more complex (see \cite{Curran-Goldmakher, Granville-Shakan, Granville-Shakan-Walker}), and it becomes less effective to determine $p_{\cA}(t)$ or to produce, in general, a suitable bound for $n_0(\cA)$. That being said, the authors of \cite{Curran-Goldmakher} give a complete solution for subsets $\cA \subset \ZZ^{n}$ of $n+2$ elements and such that the different set $\cA-\cA$ generates $\ZZ^{n}$. The structure of the $t-$fold sumsets $t\cA$ has played a significant role on establishing bounds for $n_{0}(\cA)$ when the convex hull of the subset $\cA \subset \ZZ^{n}$ is an $n-$simplex. Overall, the polynomial $p_{\cA}(t)$ and its coefficients remain less understood. 

In this paper, we stress the link between sumsets and the geometry of projective varieties, outlined for instance in \cite{Elias, Jelinek-Klazar, Miller-Sturmfels}. Our aim is twofold. First, to any finite subset $\cA \subset \ZZ^{n}$, we explicitly associate a unique projective toric variety $Y_{n,d_{\cA}}$ in $\PP^{|\cA|-1}$ whose Hilbert function $\HF_{Y_{n,d_{\cA}}}(t)$ coincides with the cardinality function $\varphi_{\cA}(t)$ for any $t \geq 0$. The variety $Y_{n,d_{\cA}}$ turns out to be a {\em monomial projection of a Veronese variety} (see Section \ref{Preliminaries}). This identification and the geometric knowledge of monomial projections of Veronese varieties allows us to go ahead with the study of the polynomial $p_{\cA}(t)$ and to provide bounds for the phase transition $n_0(\cA)$ of subsets $\cA \subset \ZZ^{n}$ in arbitrary dimension $n \geq 2$. Second, we make use of the theory of sumsets to derive new geometric information about the varieties $Y_{n,d_{\cA}}$ and, in particular of the so called $RL-$variety (see Section \ref{Section: GT-sumsets and RL-varieties.}). This relationship between additive number theory and projective geometry will allow us to go back and forth and to use algebraic and geometric results on monomial projections of Veronese varieties to recover and improve results on additive combinatorics and vice versa.

Since the cardinality of a $t$-fold sumset is invariant under translations, we can assume that $\cA \subset \ZZ_{\geq 0}^{n}$. We set $d_{\cA} := \min\{d \in \ZZ_{\geq 0} \mid \sum_{i=1}^{n} a_{i} \leq d, \, \forall a = (a_1,\hdots,a_n) \in \cA\}$ and we define:
\[\Omega_{n,d_{\cA}} := \{x_0^{d_{\cA}-a_1-\cdots-a_n}x_1^{a_1}\cdots x_n^{a_n} \mid a = (a_1,\hdots,a_n) \in \cA\} = \{m_1,\hdots,m_{|\cA|}\},\] 
a set of monomials of degree $d_{\cA}$ in $R = \kk[x_0,\hdots,x_n]$. Moreover, applying a suitable translation, we can also assume that $\GCD(m_1,\hdots,m_{|\cA|}) = 1$. We denote by $Y_{n,d_{\cA}}$ the variety image of the rational map $\PP^{n} \dashrightarrow \PP^{|\cA|-1}$ defined by the parameterization $(m_1: \hdots: m_{|\cA|})$. We establish that the homogeneous coordinate ring $A(Y_{n,d_{\cA}})$ of $Y_{n,d_{\cA}}$ is isomorphic to the semigroup ring of monomials $\kk[\Omega_{n,d_{\cA}}]$, which in turn gives a description of $A(Y_{n,d_{\cA}})$ and the homogeneous ideal $\I(Y_{n,d_{\cA}})$ of $Y_{n,d_{\cA}}$ in terms of $\cA$ (Section \ref{Section: Additive number theory and projective geometry}). 
As a result of these facts, we deduce that $\varphi_{\cA}(t) = \HF_{Y_{n,d_{\cA}}}(t)$ for any $t \geq 0$ (Proposition \ref{BasicHF}). Consequently, $\varphi_{\cA}(t)$ is a polynomial in $\Q[t]$ for $t$ sufficiently large and, hence, $p_{\cA}(t)$ is the Hilbert polynomial of the monomial projection $Y_{n,\cA}$.  Both facts provide further information on $\varphi_{\cA}(t)$ and $p_{\cA}(t)$ ((\ref{Hilbert function Macaulay}) and (\ref{Gotzmann expansion})). For instance, a new geometric description of the leading coefficient of $p_{\cA}(t)$ (Proposition \ref{Proposition: geometric interpretation degree}) and a complete solution for any subset $\cA$ of $n+1$ and $n+2$ elements, respectively, associated to an $n-$dimensional monomial projection $Y_{n,d_{\cA}}$ (Proposition \ref{Proposition: cases n+1 and n+2}), recovering the corresponding results in \cite{Curran-Goldmakher}. 
In this setting, $n_0(\cA)$ translates into the {\em regularity} of the Hilbert function $\HF_{Y_{n,d_{\cA}}}(t) = \varphi_{\cA}(t)$, thus the Castelnuovo--Mumford regularity $\reg(Y_{n,d_{\cA}})$ of $A(Y_{n,d_{\cA}})$, which we denote by $\reg(\cA)$, bounds $n_{0}(\cA) \leq \reg(\cA)+1$ (see \cite{BH93} for the definition and basic properties of the Castelnuovo-Mumford regularity). Taking advantage of this fact and the results of \cite{Herzog-Hibi, Hoa-Stuckrad}, we provide new estimations for $n_{0}(\cA)$ when the convex hull of $\cA$ is an $n-$simplex with vertexes $0,(d_{\cA},0\hdots,0), \hdots, (0,\hdots,0,d_{\cA})$, or equivalently when $\kk[\Omega_{n,d_{\cA}}]$ is the semigroup ring of a simplicial affine semigroup (Theorems \ref{Theorem: Bounds reg Hoa-Stuckrad}, \ref{Theorem: Bounds reg Hoa-Stuckrad-continuacio} and \ref{Theorem: Bounds reg Hoa-Stuckrad-reduction number}). 

The arguments developed so far do not consider any further particularity of the subset $\cA$. In this sense, the bounds established for the phase transition $n_0(\cA)$ are often very far from the actual value of $n_0(\cA)$. Motivated by this fact, in the last part of this paper we introduce $GT-$subsets and $GT-$sumsets (Definition \ref{Definition: GT-subset}). They form a family of $t$-fold sumsets uniquely determined by a linear system of congruences. The algebraic structure of $GT-$sumsets $t\cA$  as solutions of a linear system of congruences allows to significantly improve the bound for $n_0(\cA)$ (Proposition \ref{Proposition: Phase transition GT-subsets}). In addition, they provide a good field to delve into the polynomial $p_{\cA}(t)$, which is completely determined in several interesting cases (Propositions \ref{Proposition: polynomial GT-surfaces} and \ref{Proposition: polynomial d prime}). Finally, we apply this new approach and recent results on sumsets to compute the degree and the Hilbert polynomial of $RL-$varieties (Proposition \ref{Proposition: degree RL-variety}), a family of rational smooth projective varieties introduced in \cite{ThesisLiena} and intrinsically related to $GT-$subsets whose geometry is barely known. 

\medskip
The content of this paper is organized as follows. In Section \ref{Preliminaries}, we gather the basic definitions and notations needed in the body of this article. Given a finite subset $\cA \subset \ZZ^{n}$, we introduce the notion of {\em $t$-fold sumsets} $t\cA \subset \ZZ^{n}$ and the cardinality function $\varphi_{\cA}(t) = |t\cA|$. We recall Khovanskii's result on the polynomial growth of $\varphi_{\cA}(t)$, i.e. the existence of a polynomial $p_{\cA}(t) \in \Q[t]$ such that $\varphi_{\cA}(t) = p_{\cA}(t)$ for $t$ sufficiently large, along to the so called {\em phase transition} $n_0(\cA)$ of $\varphi_{\cA}$ and illustrating examples. Afterwards, we introduce Veronese varieties $X_{n,d} \subset \PP^{\binom{n+d}{n}-1}$ and {\em monomial projections} $Y_{n,d}$ of $X_{n,d}$, which play a central role through this work. 

In Section \ref{Section: Additive number theory and projective geometry}, to any subset $\cA \subset \ZZ^{n}$, we explicitly associate a monomial projection $Y_{n,d_{\cA}}$ of the Veronese variety $X_{n,d_{\cA}}$ whose Hilbert function agrees with the cardinality function $\varphi_{\cA}(t)$. This link allows us, on one hand, to describe the homogeneous coordinate ring $A(Y_{n,d_{\cA}})$ and the homogeneous ideal $\I(Y_{n,d_{\cA}})$ of $Y_{n,d_{\cA}}$ in terms of the $t$-fold sumsets $t\cA$ and, on the other hand, to interpret the function $\varphi_{\cA}(t)$ and the polynomial $p_{\cA}(t)$ in terms of the variety $Y_{n,d_{\cA}}$ (Proposition \ref{BasicHF}). Through this identification, we obtain further information on the general structure of $p_{\cA}(t)$  and we provide a combinatorial and a geometric formula for the leading coefficient of $p_{\cA}(t)$ (Proposition \ref{Proposition: geometric interpretation degree} and (\ref{Equation: degree from lattice ideal})). Using purely geometric techniques, we determine $p_{\cA}(t)$ and $n_0(\cA)$ when $\cA$ contains $n+1$ and $n+2$ elements and $Y_{n,d_{\cA}}$ is an $n-$dimensional  projective variety, recovering some of the results in \cite{Curran-Goldmakher}. 

In Section \ref{sec-phase}, we focus our attention on subsets $\cA \subset \ZZ^{n}$ whose associated convex hull $\conv(\cA)$ is an $n-$simplex (Definition \ref{Definition: A with n-simplex}) and we gather the bounds known so far for $n_0(\cA)$ in this case (Theorem \ref{Theorem: effectives khovanskii} and Proposition \ref{Proposition: boudn K(A,B)}). 
Using the Castelnuovo--Mumford regularity $\reg(\cA)$ of $A(Y_{n,d_{\cA}})$, we provide improved bounds for $n_0(\cA)$ under some technical assumptions on the convex hull of $\cA$ (Theorems \ref{Theorem: Bounds reg Hoa-Stuckrad}, \ref{Theorem: Bounds reg Hoa-Stuckrad-continuacio} and \ref{Theorem: Bounds reg Hoa-Stuckrad-reduction number}). We end this section with Example \ref{Example: Section 4 compare} where we compare our results with previous bounds for the phase transition $n_{0}(\cA)$.

In Section \ref{Section: GT-sumsets and RL-varieties.}, we focus our attention on how far are the bounds for $n_0(\cA)$ from its real value. We introduce the notions of $GT-$subsets and $GT-$sumsets (Definition \ref{Definition: GT-subset}), they are uniquely determined by the $\ZZ_{\geq 0}^{n+1}-$solutions of linear systems of congruences. We show that this algebraic property allows to improve significantly the bound for $n_0(\cA)$ (Proposition \ref{Proposition: Phase transition GT-subsets}) and, in several cases, to actually compute the polynomial $p_{\cA}(t)$ (Propositions \ref{Proposition: polynomial GT-surfaces} and \ref{Proposition: polynomial d prime}). Finally, we use the connection between sumsets and geometry to estimate the Castelnuovo--Mumford regularity and to compute the degree and the Hilbert Polynomial of $RL-$varieties (Theorem \ref{Proposition: degree RL-variety}), a family of smooth rational monomial projections of $X_{n,d_{\cA}}$ introduced in \cite{ThesisLiena}. $RL-$varieties are actually monomial projections image of embeddings of $\PP^{n}$. Using this fact along with the combinatorial structure of the subset defining them, in \cite{ThesisLiena} it is determined the cohomology of the normal bundle of any $RL-$variety.  

\medskip
\noindent {\bf Notation.} Throughout this paper $\kk$ will be
 an algebraically closed field of characteristic zero and we set $R:=\kk[x_0,\hdots ,x_n].$ 
  Let $A=\sum _{i\ge 0} A_{i}$ be an standard $\kk = A_0$ algebra, i.e. $A = \kk[A_1]$. We denote by $\HF_A$ the Hilbert function of $A$, i.e. $\HF_A(i) =
\dim_{\kk} A_{i}$ for all $i\ge  0$. It is well known that there exists a rational coefficient polynomial $\HP_{A}$, the
Hilbert polynomial of $A$, such that $\HP_A(i) = \HF_A(i)$ for $i\gg  0$.
 
Given  integers $n, d\ge 1$, we define the  {\em $d-$th binomial expansion of $n$} as
$$
n=\binom{m_d}{d}+\cdots +\binom{m_e}{e}
$$
where $m_d>\cdots >m_e\ge e\ge 1$ 
are uniquely determined integers (see \cite[Lemma 4.2.6]{BH93}).
We write
$$
n^{<d>}=\binom{m_d+1}{d+1}+\cdots +\binom{m_e+1}{e+1}.
$$

%%%%%%%%%%%%%%%%%%%%%%%%%%%%%%%%%%%%%%%%%%%%%%%%%%%%%%%%%%%%%%%%%%%%%%%%%%%%%%%%%%%%%%%%%%%%%%%%%%%%%%%%%%%%%%%%%%%%%%%%%%%%
%%%%%%%%%%%%%%%%%%%%%%%%%%%%%%%%%%%%%%%%%%%%%%%%%%%%%%%%%%%%%%%%%%%%%%%%%%%%%%%%%%%%%%%%%%%%%%%%%%%%%%%%%%%%%%%%%%%%%%%%%%%%
%%%%%%%%%%%%%%%%%%%%%%%%%%%%%%%%%%%%%%%%%%%%%%%%%%%%%%%%%%%%%%%%%%%%%%%%%%%%%%%%%%%%%%%%%%%%%%%%%%%%%%%%%%%%%%%%%%%%%%%%%%%%
%%%%%%%%%%%%%%%%%%%%%%%%%%%%%%%%%%%%%%%%%%%%%%%%%%%%%%%%%%%%%%%%%%%%%%%%%%%%%%%%%%%%%%%%%%%%%%%%%%%%%%%%%%%%%%%%%%%%%%%%%%%%
%%%%%%%%%%%%%%%%%%%%%%%%%%%%%%%%%%%%%%%%%%%%%%%%%%%%%%%%%%%%%%%%%%%%%%%%%%%%%%%%%%%%%%%%%%%%%%%%%%%%%%%%%%%%%%%%%%%%%%%%%%%%
\medskip
\section{Preliminaries}\label{Preliminaries}  In this section, we gather the main notations, definitions and results we use in this paper.  The reader can look at \cite{Freiman, Geroldinger-Ruzsa, Nathanson} for more details. 

\subsection{Sumsets} Let $n \geq 1$ be an integer and $\cA \subset \ZZ^{n}$ a non-empty finite subset. For any $t\in \NN$, a {\em $t$-fold sumset} $t\cA$ is defined as follows:
 $$t\cA =\{ a_1+\cdots +a_t \mid a_{i}\in \cA \text{ for all }
1\le i \le t \}.$$
As usual, we set $0\cA=\{0\}$. A longstanding problem in additive combinatorics is to determine the cardinality $|t\cA|$ of the $t$-fold sumset $t\cA$ as $t$ grows. To this end, we introduce the {\em cardinality function}
 $$ \varphi _{\cA}: \NN \longrightarrow \NN, \quad \quad t \mapsto |t\cA|.$$
\begin{proposition}{(\cite[Theorem 1]{Khovanskii})} Let $\cA \subset \ZZ^n$ be a non empty finite subset. There exists a polynomial $p_{\cA}(t) \in \Q[t]$ such that $\varphi_{\cA}(t) = p_{\cA}(t)$ for t sufficiently large. The {\em phase transition} of the cardinality function $\varphi_{\cA}(t)$ is defined as 
 $$n_0(\cA):=\min \{n_0\in \ZZ_{\geq 0} \mid \varphi_{\cA}(t) = p_{\cA}(t), \,\forall t \geq n_0\}.$$
\end{proposition} 
The coefficients of the polynomial $p_{\cA}(t)$ and the value of $\varphi_{\cA}(t)$ for small $t$ are barely known. Another interesting problem is to determine the phase transition $n_0(\cA)$ of $\varphi_{\cA}$. The case $n=1$ has lately received a lot of attention and our goal is to address the general case, i.e. $n \ge 1$. First, we deal with arbitrary subsets $\cA \subset \ZZ^{n}$ and, later in Section \ref{Section: GT-sumsets and RL-varieties.}, we restrict our attention to suitable subsets ($GT-$subsets) to improve previous results.

Let us start with easy examples which show  that the behaviour of $|t\cA|$ for small $t$, the coefficients of $p_{\cA}(t)$ and the phase transition strongly depend on the structure of $\cA$.

\begin{example} \rm \label{firstexample} (i) We consider the subset $\cA^1 =\{(0,0),(3,0),(2,2),(0,1)\} \subset \ZZ^{2}$. We have $|\cA^1 |=4$, $|2\cA^1|=10$, $|3\cA^1|=20$, $|4\cA^1|=35$ and $|t\cA^1|=4t^2-16t+36$ for all $t\ge 5$. Therefore, the phase transition is $5$.

\noindent (ii) We consider now the set $\cA^2=\{(0,0),(2,0),(2,2),(0,1)\} \subset \ZZ^{2}$. We have $|\cA^2|=4$, $|2\cA^2|=10$, $|3\cA^2|=19$, $|4\cA^2|=31$ and $|t\cA^2|=\frac{3}{2}(t^2+t)+1$ for $t\ge 0$. 
Therefore, the phase transition is 0.

Notice that $|\cA^1|=|\cA^2|$, and $\cA^1$, $\cA^2$ differ only by one element. Nevertheless, the behaviour of the $t-$fold sumsets $|t\cA^1|$ and $|t\cA^2|$ and the phase transition drastically change. This phenomenon  will be explained geometrically in next sections.
\end{example}

Next goal is to identify $p_{\cA}(t)$ with the Hilbert polynomial of a 
suitable monomial projection $Y_{n,d}$ of the Veronese variety $X_{n,d}$;  and use the geometry of Veronese varieties and their monomial projections to determine upper bounds for the phase transition as well as for identifying certain coefficients of $p_{\cA}(t)$. For sake of completeness we recall below the basic facts on Veronese varieties.

\subsection{Veronese varieties} 
 We fix integers $n,d \geq  1$ and we set $N_{n,d}:=\binom{n+d}{n}$. We consider the set  $\cM_{n,d} = \{m_{0},\hdots, m_{N_{n,d}-1}\} \subset R$  of {\em all} monomials of degree $d$ in $R$, ordered lexicographically. The {\em Veronese variety} $\cV_{n,d} \subset \PP^{N_{n,d}-1}$ is defined as the image of the Veronese embedding of $\PP^{n}$
\[\nu_{n,d}: \PP^{n} \longrightarrow \PP^{N_{n,d}-1}\]
which sends a point $p = (x_0:\cdots:x_n) \in \PP^{n}$ to $\nu_{n,d}(p) = (m_{0}(p):\cdots : m_{N_{n,d}-1}(p)) \in \PP^{N_{n,d}-1}$ (for further details see, for instance, \cite{Ha}).

We take new variables $w_0, \hdots, w_{N_{n,d}-1}$ and we set $S := \kk[w_0, \hdots, w_{N_{n,d}-1}]$. The homogeneous ideal $\I(\cV_{n,d})\subset S$ of the Veronese variety $\cV_{n,d} \subset \PP^{N_{n,d}-1}$ is the homogeneous binomial prime ideal generated by all binomials of degree $2$ of the form:
\begin{equation}\label{Equation: generators I(Veronese)}
w_{i}w_{j}-w_{\ell }w_{k}  \quad \text{such that} \quad  m_{i}m_{j}=m_{\ell }m_{k}.
\end{equation}

\begin{example}\rm
We take an integer $d \geq 1$ and we set $S = \kk[w_0,\hdots, w_{d}]$. The {\em rational normal curve of degree $d$} is the Veronese curve $\cV_{1,d} \subset \PP^{d}$. It is the image of the morphism
\[\nu_{1,d}: \PP^{1} \longrightarrow \PP^{d}, \quad \nu_{1,d}(x_0:x_1) = (x_0^d:x_0^{d-1}x_1:\cdots:x_0x_1^{d-1}:x_1^{d}).\]
The homogeneous ideal $\I(\cV_{1,d}) \subset S$ of $\cV_{1,d}$ is the binomial prime ideal generated by the $\binom{d}{2}$ quadrics obtained from the $2 \times 2$ minors of the matrix
\[\left(\begin{array}{llll}
w_0 & w_1 & \cdots & w_{d-1}\\
w_1 & w_2 & \cdots & w_{d}
\end{array}\right).\] 
These quadrics correspond to the $2 \times 2$ minors of the matrix of rank $1$
\[\left(\begin{array}{llll}
x_0^d & x_0^{d-1}x_1 & \cdots & x_0x_1^{d-1}\\
x_0^{d-1}x_1 & x_0^{d-2}x_1 & \cdots & x_1^d
\end{array}\right).\] 
An easy computation shows that the Hilbert polynomial of $X_{1,d}$ is:
$$\HF_{X_{1,d}}(t):=\dim (S/\I(X_{1,d}))_t=td+1=\HP_{X_{1,d}}(t) \text{ for all } t\ge0.
$$
\end{example}

\medskip
  Given a subset $\Omega_{n,d} = \{m_{i_0}, \hdots, m_{i_{\mu_{n,d}-1}}\} \subseteq \cM_{n,d}$ of $\mu_{n,d}=|\Omega_{n,d} |$ monomials, we denote by 
  \[\varphi_{\Omega_{n,d}}: \PP^{n} \dashrightarrow \PP^{\mu_{n,d}-1}\]
  the rational map defined by $\Omega_{n,d}$ which sends a point $p = (x_0: \cdots: x_n) \in \PP^{n}$ to $\varphi_{\Omega_{n,d}}(p) = (m_{i_0}(p): \cdots : m_{i_{\mu_{n,d}-1}}(p)) \in \PP^{\mu_{n,d}-1}.$ We say that $Y_{n,d} := \overline{\varphi_{\Omega_{n,d}}(\PP^{n})} \subset \PP^{\mu_{n,d}-1}$ is the {\em monomial projection of the Veronese variety $\cV_{n,d}$ parameterized by $\Omega_{n,d}$.}  So, we have the commutative diagram:

\begin{center}
\begin{tikzpicture}
\matrix (M) [matrix of math nodes, row sep=3em, column sep=3em, nodes={minimum height = 1cm, minimum width = 1cm, anchor=center}]{
\PP^{n} & X_{n,d} \\
        & Y_{n,d} \\
};
\draw[->] (M-1-1) -- (M-1-2) node[midway,above]{$\nu_{n,d}$};
\draw[->] (M-1-2) -- (M-2-2) node[midway,right]{$\pi$};
\draw[->,dashed] (M-1-1) -- (M-2-2) node[midway,below]{$\varphi_{\Omega_{n,d}}\hspace{0.5cm}$ };
\end{tikzpicture}
\end{center}
where $\pi$ is the projection of the Veronese variety $X_{n,d} \subset \PP^{N_{n,d}-1}$ from the linear subspace generated by the coordinate points $(0:\cdots:0:1:0:\cdots:0) \in \PP^{N_{n,d}-1}$ with $1$ in position $i$ such that $m_i \notin \Omega_{n,d}$  to the linear subspace $V(w_{m_{i}}, \; m_{i}  \notin \Omega_{n,d}) \cong \PP^{\mu_{n,d}-1}  \subset \PP^{N_{n,d}-1}$.
In particular, $Y_{n,d} \subset \PP^{\mu_{n,d}-1}$ is called a {\em simple} (resp. {\em double}) monomial projection if $\Omega_{n,d}$ is obtained from $\cM_{n,d}$ by deleting only one monomial (resp. two monomials).
 
 \begin{example}\rm 
 We take $n = 1$, $d = 4$ and $\Omega_{1,4} = \{x_0^4,x_0^3x_1,x_0x_1^3,x_1^4\} \subset \kk[x_0,x_1]$. The simple monomial projection $Y_{1,4} \subset \PP^{3}$ parameterized by $\Omega_{1,4}$ is the rational quartic in $\PP^{3}$ obtained as the monomial projection of the rational normal curve $X_{1,4}$ of degree $4$ in $\PP^{4} = \proj(\kk[w_0,w_1,w_2,w_3,w_4])$ from the coordinate point $(0:0:1:0:0)$ to the hyperplane $V(w_2) \subset \PP^{4}$. 
\end{example}

%%%%%%%%%%%%%%%%%%%%%%%%%%%%%%%%%%%%%%%%%%%%%%%%%%%%%%%%%%%%%%%%%%%%%%%%%%%%%%%%%%%%%%%%%%%%%%%%%%%%%%%%%%%%%%%%%%%%%%%%%%%%
%%%%%%%%%%%%%%%%%%%%%%%%%%%%%%%%%%%%%%%%%%%%%%%%%%%%%%%%%%%%%%%%%%%%%%%%%%%%%%%%%%%%%%%%%%%%%%%%%%%%%%%%%%%%%%%%%%%%%%%%%%%%
%%%%%%%%%%%%%%%%%%%%%%%%%%%%%%%%%%%%%%%%%%%%%%%%%%%%%%%%%%%%%%%%%%%%%%%%%%%%%%%%%%%%%%%%%%%%%%%%%%%%%%%%%%%%%%%%%%%%%%%%%%%%
%%%%%%%%%%%%%%%%%%%%%%%%%%%%%%%%%%%%%%%%%%%%%%%%%%%%%%%%%%%%%%%%%%%%%%%%%%%%%%%%%%%%%%%%%%%%%%%%%%%%%%%%%%%%%%%%%%%%%%%%%%%%
%%%%%%%%%%%%%%%%%%%%%%%%%%%%%%%%%%%%%%%%%%%%%%%%%%%%%%%%%%%%%%%%%%%%%%%%%%%%%%%%%%%%%%%%%%%%%%%%%%%%%%%%%%%%%%%%%%%%%%%%%%%%
\section{Sumsets and monomial projections of Veronese varieties}
\label{Section: Additive number theory and projective geometry}
The goal of this section is to associate to {\em any} finite subset $\cA \subset \ZZ^{n}$ a monomial projection  $Y_{n,d_{\cA}}$ of the Veronese variety $X_{n,d_{\cA}}$ whose Hilbert function $\HF_{Y_{n,d_{\cA}}}(t)$ models the cardinality function $\varphi_{\cA}(t) = |t\cA|$ and, hence, allows to conclude that there is  a polynomial $p_{\cA}(t)\in \Q[t]$  of degree the dimension of $Y_{n,d_{\cA}}$ such that $\varphi_{\cA}(t)= p_{\cA}(t)$ for $t$ sufficiently large. 

\begin{definition}\label{Definition: sets A}
\rm 
 For any integer $n \geq 1$  and any finite subset $\cA \subset \ZZ^{n}_{\geq 0}$,  we denote by $d_{\cA}$ the minimum of the integers $\delta \geq 0$ such that any element $a = (a_1,\hdots, a_n) \in \cA$ satisfies $|a| = a_1 + \cdots + a_n \leq \delta$.
We define $Y_{n,d_{\cA}} \subset \PP^{|\cA|-1}$ the  monomial projection of the Veronese variety $X_{n,d_{\cA}} \subset \PP^{N_{n,d_{\cA}}-1}$ parameterized by the set of monomials
\[\Omega_{n,d_{\cA}} = \{x_0^{d_{\cA}-|a|}x_1^{a_1}\cdots x_{n}^{a_{n}} \mid a \in \cA\}.\]
\end{definition}
For simplicity, we use the notation $Y_{n,d_{\cA}}$, notwithstanding the reader has to be aware of the fact that $Y_{n,d_{\cA}}$ depends on  $\cA $ and not just on $d_{\cA}$.
\begin{remark}\label{Remark: tau}\rm 
Given a finite subset $\cA \subset \ZZ^{n}$, there is a unique translation $\tau: \ZZ^{n} \to \ZZ^{n}$ such that $\tau(\cA) \subset \ZZ^{n}_{\geq 0}$ and $\GCD(m \in \Omega_{n, d_{\tau(\cA)}}) = 1$. 
In fact, $\tau$ is the translation defined by $-C$ where the $i$-th component of $C$ is the minimum of the $i$-th components of the elements $a\in \cA$. 
We have: 
$$\varphi_{\cA}(t) = |t\cA| = |t\tau(\cA)| =  \varphi_{\tau(\cA)}(t)$$ 
for all $ t \geq 0.$
Hence, without loss of generality  we will assume in the sequel that $\cA  \subset \ZZ^{n}_{\geq 0}$
and $\GCD(m \in \Omega_{n, d_{\cA}}) = 1$.
\end{remark}

%To any subset $\cA \subset \ZZ^{n}_{\ge 0}$, we associate the monomial projection of the Veronese %variety $X_{n,d_{\tau(\cA)}}$ parameterized by $\Omega_{n,d_{\tau(\cA)}}$, we denote it by %$Y_{n,d_{\cA}}$. We will prove that the Hilbert function of $Y_{n,d_{\cA}}$ is the function %$\varphi_{\cA}(t) = \varphi_{\tau(\cA)}(t)$ (Proposition \ref{BasicHF}).  

 Given an integer $n\ge 1$ and a finite subset  $\cA \subset \ZZ_{\geq 0}^{n}$,  we consider $\Omega_{n,d_{\cA}} = \{m_1,\hdots,m_{|\cA|}\} \subset R$ the set of monomials determined by $\cA$. We take $w_1,\hdots,w_{|\cA|}$ new variables and $S = \kk[w_1,\hdots,w_{|\cA|}]$. 
 The homogeneous ideal $\I(Y_{n,d_{\cA}}) \subset S$ of $Y_{n,d_{\cA}}$ is the kernel of the epimorphism 
\[\rho: S \to \kk[\Omega_{n,d_{\cA}}], \quad \rho(w_i) = m_i, \; i = 1,\hdots, |\cA|.\]
It is a binomial prime ideal of $S$ generated by (\cite[pag 335]{Hochster}):
\begin{equation}\label{Equation: binomials}
\left\{\prod_{i=1}^{|\cA|} w_i^{\alpha_i} - \prod_{i=1}^{|\cA|} w_i^{\beta_i} \left| \; \prod_{i=1}^{|\cA|} m_i^{\alpha_i} = \prod_{i=1}^{|\cA|} m_i^{\beta_i}, \; \alpha_i,\beta_i \in \ZZ_{\geq 0}\right. \right\}    
\end{equation}
and the homogeneous coordinate ring $A(Y_{n,d_{\cA}}) := S/\I(Y_{n,d})$ of $Y_{n,d_{\cA}}$ is isomorphic to $\kk[\Omega_{n,d_{\cA}}]$. 

To simplify, we denote by $\HF_{\cA}$ the Hilbert function of  $A(Y_{n,d_{\cA}})$ (see \cite[Chapter 4]{BH93} for further details).
Recall that for any integer $t \geq 0$, 
 $\HF_{\cA}(t)$ equals to  $\dim_{\kk} A(Y_{n,d_{\cA}})_t=\dim_{\kk} \kk[\Omega_{n,d_{\cA}}]_{td}$. Moreover, we have:

\begin{proposition}
\label{BasicHF}
For any integer $n\ge 1$ and any finite subset  $\cA \subset \ZZ_{\geq 0}^{n}$, it holds:
\begin{enumerate}
    \item[(1)]
$\HF_{\cA}(t) = \varphi_{\cA}(t) = |t\cA|$ for all $t\ge 0$.
\item[(2)] \cite[Theorem 4.1.3]{BH93}
There exists a polynomial $\HP_{\cA}(t)=p_{\cA}(t) \in \Q[t]$ of degree $r = \dim(Y_{n,d_{\cA}})\le n$,  the Hilbert polynomial of $Y_{n,d_{\cA}}$, such that
$\HP_{\cA}(t)=p_{\cA}(t)$ for $t$ sufficiently large. 
\end{enumerate}
\end{proposition}
\noindent
The {\em degree} $\deg(Y_{n,d_{\cA}})$ of $Y_{n,d_{\cA}}$ is defined algebraically as $r!$ times the leading coefficient of $\HP_{\cA}(t)$. It corresponds geometrically to the number of points of intersection of $Y_{n,d_{\cA}}$ with a sufficiently general linear subspace of $\PP^{|\cA|-1}$ of dimension $|\cA| - r - 1$ (see, for instance, \cite[Chapter \textsc{i} \S 7]{Hartshorne}). This perspective provides that for $t$ large enough, $\varphi_{\cA}(t) = |t\cA|$ is a polynomial $p_{\cA}(t) \in \Q[t]$ of degree $r$ and, hence, a geometrical interpretation of its leading coefficient $\frac{\deg(Y_{n,d_{\cA}})}{r!}$. 
In addition, the phase transition $n_{0}(\cA)$, known also as the regularity of the Hilbert function, is bounded by the Castelnuovo-Mumford regularity $\reg(\cA)$ plus one (see also \cite{BH93} for further details).

\medskip
The numerical functions $H:\mathbb N \longrightarrow \mathbb N$ that are Hilbert functions of standard $\kk$-algebras were characterized by Macaulay, \cite{BH93}. Indeed, given a numerical function $H:\mathbb N \longrightarrow \mathbb N$ the following conditions are equivalent:
\begin{enumerate}
    \item[(i)] There exists a $\kk$-algebra $A$ such that $H = \HF_A$,
    \item[(ii)]  $H(0)=1$ and $H(t+1)\le H(t)^{<t>}$ for all $t \ge 1$.
\end{enumerate}

\noindent 
Hence, for any finite subset $\cA\subset \ZZ_{\ge 0}^n$, the growth of the cardinality function $\varphi _{\cA}(t)=|t\cA |$  satisfies:
\begin{equation}\label{Hilbert function Macaulay}
\varphi _{\cA}(t+1)\le \varphi _{\cA}(t)^{<t>}\quad  \text{ for all } t\ge 1.
\end{equation}

The polynomials $p(t)\in \mathbb Q[t]$ that are Hilbert polynomials are characterized 
as those admitting a Gotzmann development (see \cite{Gotz} and \cite{Green}).
Recall that  $p(t)  \in \mathbb Q[t]$ admits a Gotzmann development if $p=0$ or there exist integers $a_1\ge \cdots \ge a_s\ge 0$ such that
\begin{equation}\label{Gotzmann expansion}
p(t)=\binom{t+a_1}{a_1}+\binom{t+a_2-1}{a_2}+\cdots + \binom{t+a_s-(s-1)}{a_s}.
\end{equation}
Moreover, this representation is unique. From \cite[Theorem 4.4]{Blanc} we have that the following conditions are equivalent:
\begin{enumerate}
    \item[(i)]  $p(t)\in \mathbb Q[t]$ is the Hilbert polynomial of a standard $\kk$-algebra $A$
    \item[(ii)] $p(t)$ admits a Gotzmann development.
\end{enumerate}
The integer $s$ is an upper bound of the Castelnuvo-Mumford regularity $\reg(A)$ of $A$ (see \cite{BH93}).
In general this bound is far to be optimal. In fact, the  Eisenbud-Goto conjecture claims that if $A$ is an integral $\kk$-algebra then
$$
\reg(A)\le \deg(A)-\codim(A)+1.
$$
Although this bound has been disproved (\cite{MacPee}), it is true for several types of rings covering some of the considered in this paper, see \cite{Nit} and the references therein. 
See also Section \ref{sec-phase} for further results on the phase transition of finite sets $\cA \subset \ZZ_{\geq 0}^{n}$.

\medskip
Since $Y_{n,d_{\cA}}$ is a monomial projection of the $n$-dimensional Veronese variety $X_{n,d_{\cA}}$, its dimension is bounded by $n$ and its degree by the degree $d_{\cA}^n$ of $X_{n,d_{\cA}}$. 
However, the following examples show that both the dimension and the degree of the variety $Y_{n,d_{\cA}}$ can be smaller than those of the Veronese variety $X_{n,d_{\cA}}$. 

\begin{example}\rm \label{ex2} (i) We take integers $n,d \geq 1$ and $\cA = \{(a_1,\hdots,a_n) \in \ZZ_{\geq 0}^{n} \mid a_1+\cdots+a_n\leq d\}.$ We  have $d_{\cA}=d$, $|\cA |=N_{n,d}$ and we observe that $\Omega_{n,d_{\cA}}$ is the set of all monomials of degree $d$ in $R$. Therefore, $Y_{n,d_{\cA}}$ is the Veronese variety $X_{n,d} \subset \PP^{N_{n,d}-1}$.

\noindent (ii) We take $n = 2$ and $\cA = \{(0,0), (1,0), (0,1), (2,0), (0,2), (3,0), (2,1), (1,2), (0,3)\}\subset \ZZ^2_{\ge 0}$. We have $d_{\cA}=3$, $|\cA|=9$ and the associated monomial projection $Y_{2,3} \subset \PP^{8}$ of the Veronese surface $X_{2,3} \subset \PP^{9}$ is the surface parameterized by 
$\Omega_{2,3} = \{x_0^3, x_0^2x_1, x_0^2x_2, x_0x_1^2, x_0x_2^2, x_1^3, x_1^2x_2, x_1x_2^2, x_2^3\}$. It is a simple monomial projection of $X_{2,3}$. 

\noindent (iii) We take $n = 2$ and $\cA = \{(3,1), (2,2), (1,3), (0,4)\}\subset \ZZ^2_{\ge 0}$.  We have $d_{\cA}=|\cA|=4$ and the associated monomial projection $Y_{2,4} \subset \PP^{3}$ of $X_{2,4} \subset \PP^{14}$ is a curve parameterized by 
$\Omega_{2,4} = \{x_1^3x_2, x_1^2x_2^2, x_1x_2^3, x_2^4\}$.

\noindent (iv) We come back to Examples \ref{firstexample} (i) and (ii).
We take $\cA ^1= \{(0,0),(3,0),(2,2),(0,1)\}\subset \ZZ^2_{\ge 0}$. We have $d_{\cA^1} = |\cA^1|=4$ and the associated monomial projection $Y_{2,4}^1 \subset \PP^{3}$ of $X_{2,4} \subset \PP^{14}$ is the surface of degree 8 parameterized by 
$\Omega_{2,4}^1 = \{x_0^4, x_0^3x_2, x_0x_1^3, x_1^2x_2^2\}$. If we fix coordinates $w_0,w_1,w_2,w_3$ in $\PP^3$, the equation of $Y_{2,4}^1$ is: $w_0^5w_3^3-w_2^2w_1^6$ (see (\ref{Equation: binomials})). 
We slightly modify $\cA^1$ and we take 
 $\cA^2 = \{(0,0),(2,0),(2,2),(0,1)\}\subset \ZZ^2_{\ge 0}$. We have $d_{\cA^2} = |\cA^2| = 4$ and the associated monomial projection $Y_{2,4}^2 \subset \PP^{3}$ of $X_{2,4} \subset \PP^{14}$ is the cubic surface parameterized by 
$\Omega_{2,4}^2 = \{x_0^4,x_0^3x_2, x_0^2x_1^2, x_1^2x_2^2\}$. It has equation: $w_0^2w_3-w_2w_1^2$ (see (\ref{Equation: binomials})).
\end{example}

\medskip
From now onward, we restrict our attention to finite subsets $\cA  \subset \ZZ_{\geq 0}^{n}$ associated to  $n-$dimensional monomial projections $Y_{n,d_{\cA}}$ of $X_{n,d_{\cA}}$. This restriction is quite natural and well controlled since we have the following, \cite{Khovanskii}, 

\begin{proposition}\label{Proposition: dimension varietat} Let $\cA \subset \ZZ_{\geq 0}^n$ be a finite set. Then, the associated variety $Y_{n,d_{\cA}}$ has dimension $n$ if and only if $\ZZ(\cA-\cA)$ has maximal rank.
\end{proposition}
\begin{remark}\rm 
We also have a geometrical version of Proposition \ref{Proposition: dimension varietat}. Indeed, by \cite[Pag. 186]{Beltrametti-Carletti-Gallarati-Monti}, given   $\cA \subset \ZZ_{\geq 0}^{n}$ a finite set, then the associated variety $Y_{n,d_{\cA}}\subset \PP^{|\cA |-1}$ has dimension $n$ if and only if $\rank ((\partial m_i/\partial x_j))_{\substack{i=1, \hdots,|\cA | \\ j=0, \hdots, n}}=n+1$, where $\{m_i \}_{i=1, \hdots, |\cA |}$ are the monomials parameterizing $Y_{n,d_{\cA}}$.
\end{remark}

  It immediately follows that for any finite set $\cA \subset \ZZ_{\geq 0}^{n}$ associated to an $n-$dimensional monomial projection $Y_{n,d_{\cA}}$ of $X_{n,d_{\cA}}$ we have $|\cA |\ge n+1$. Our first goal is to study $\varphi_{\cA}(t) = |t\cA|$ for small values of $|\cA|$, i.e.  $|\cA|=n+1$ and $|\cA| = n+2$ and, to provide easier and new proofs of known formulae. To this purpose, we need to fix some notation.
   We denote by  $\conv(\cA)$  the convex hull of $A$ and let $[\ZZ^{n}:\ZZ(\cA-\cA)]$ be the index of of the subgroup $\ZZ(\cA-\cA)$ in $\ZZ^{n}$.
 It holds:
  
  \begin{proposition}\label{Proposition: cases n+1 and n+2}  Let  $\cA \subset \ZZ_{\geq 0}^{n}$ be a finite set associated  to an  $n-$dimensional projective variety $Y_{n,d_{\cA}}$ of degree $d$.
 \begin{itemize}
      \item[(i)]
  If $|\cA |= n+1$ then
  $$\varphi _{\cA}(t)=p_{\cA}(t)=\binom{t+n}{n} \text{ for all }  t\ge 0.
  $$
  In particular, $n_0(\cA)=0.$
  \item[(ii)]  If $|\cA |= n+2$, then $n_0(\cA) = d-n-1$ and  
  \[\varphi_{\cA}(t) = \left\{\begin{array}{lll}
\binom{t+n+1}{n+1} & \text{if} & 0 \leq t < d-n-1 \\[0.3cm]
\binom{t+n+1}{n+1} - \binom{t - d + n + 1}{n+1} & \text{if} & t \geq d-n-1.
\end{array}\right.\]
Moreover, we have
$$d=\deg(Y_{n,d_{\cA}})=n! \frac{\vol(\conv(\cA))}{[\ZZ^{n}:\ZZ(\cA-\cA)]}.
$$
 \end{itemize}
   \end{proposition}
  
\begin{proof} (i) By hypothesis $\cA$ defines a rational map
$\psi :\PP^n  \dashrightarrow \PP^n$ and the  closure $Y_{n,d_{\cA}}$ of its image is an $n$-dimensional subvariety of $\PP^n$. Therefore, $Y_{n,d_{\cA}}=\PP^n$ and 
$$\varphi _{\cA}(t)=p_{\cA}(t)=\HF _R(t)=\binom{t+n}{n} \text{ for all }  t\ge 0.
$$

\noindent
(ii) In this case, $\cA$ defines a rational map
$ \psi :\PP^n  \dashrightarrow \PP^{n+1}$ and the closure $Y_{n,d_{\cA}}$ of its image is a hypersurface of degree $d$ of $\PP^{n+1}$ defined by $\I(Y_{n,d_{\cA}}) = (F)$. Using the exact sequence
$$ 0 \longrightarrow S(-d) \xrightarrow{\times F} S \longrightarrow S/\I(Y_{n,d_{\cA}})\longrightarrow 0,
$$ where $\times F: S(-d) \to S$ denotes the multiplication map by $F$ and $S=\kk[w_0,\hdots ,w_{n+1}]$, 
we get the claim. The last equality follows from the fact that in \cite{Khovanskii} and \cite{Lee} it is  established that if $\ZZ(\cA-\cA)$ has maximal rank, then the leading coefficient of the polynomial $p_{\cA}(t) = \varphi_{\cA}(t)$ is
$$
 \frac{\vol(\conv(\cA))}{[\ZZ^{n}:\ZZ(\cA-\cA)]}.   
$$
\end{proof}

\begin{remark}\rm
Notice that if $|\cA|=n+2$ and $\cA-\cA$ generates $\ZZ^{n}$
 we easily recover  \cite[Theorem 1.2]{Curran-Goldmakher}:  
\[\varphi_{\cA}(t) = \left\{\begin{array}{lll}
\binom{t+n+1}{n+1} & \text{if} & 0 \leq t \leq n!\vol(\conv(\cA)) - n-2\\[0.3cm]
\binom{t+n+1}{n+1} - \binom{t - n!\vol(\conv(\cA)) + n + 1}{n+1} & \text{if} & t \geq n!\vol(\conv(\cA)) - n-1.
\end{array}\right.\]
In particular, $n_0(\cA) = n!\vol(\conv(\cA)) - n-1$.
\end{remark}

To explicitly determine the function $\varphi_{\cA}(t)$, the coefficients of the  polynomial $p_{\cA}(t)$ and the phase transition $n_{0}(\cA)$ of $\varphi_{\cA}(t)$ for arbitrary finite subsets $\cA \subset \ZZ^{n}$ with more than $n+2$ elements is out of reach. In the remaining part of this section, we will focus our attention on the leading coefficient of the polynomial $p_{\cA}(t)$.

So far we have  a description of the degree $\deg(Y_{n,d_{\cA}})$ of $Y_{n,d_{\cA}}$ in terms of the set $\cA$  and the difference set $\cA-\cA$. By  \cite{Khovanskii} and \cite{Lee}, if $ \ZZ(\cA-\cA)$ has maximal rank, then
 \begin{equation}\label{Equation: Khovanskii lead term} \deg(Y_{n,d_{\cA}})=n!\frac{\vol(\conv(\cA))}{[\ZZ^n:\ZZ(\cA-\cA)]}.
 \end{equation}
On the other hand, since $Y_{n,d_{\cA}}$ is a toric variety, $\deg(Y_{n,d_{\cA}})$ can  also be described combinatorially as follows. From now on, given a finite subset $\cA \subset \ZZ^{n}$, we set 
\begin{equation}\label{Equation: clousure of A}
\overline{\cA} = \{(d_{\cA }-|a|,a_1,\hdots,a_n) \mid a \in \cA\}\subset \ZZ_{\geq 0}^{n+1}.
\end{equation}
We denote by $M$ the $(n+1) \times |\cA|$ matrix whose columns correspond to the points of $\overline{\cA}$. By \cite[Theorem 2.13 and 4.5]{Carroll-Planas-Villarreal}, we have: 
\begin{equation}\label{Equation: degree from lattice ideal}
\deg(Y_{n,d_{\cA}}) = \frac{r!\vol(\conv(\overline{\cA}))}{\Delta_{r}}
\end{equation}
where $r = \rk(M)$, $\vol(\conv(\overline{\cA}))$ is the volume of the convex hull of $\overline{\cA} \cup \{ 0\}$ and $\Delta_{r}$ is the greatest common divisor of all the non-zero $r \times r$ minors of $M$. 

\medskip
Expressions (\ref{Equation: Khovanskii lead term}) and (\ref{Equation: degree from lattice ideal}) provide two different ways to determine the degree of $Y_{n,d_{\cA}}$ in terms of subsets. Both involve the volume of convex polyhedrons and the {\em Smith normal form} of certain matrix (see, for instance, \cite{Carroll-Planas-Villarreal}). \iffalse However, $\vol(\conv(\overline{\cA}))$ could be more difficult to  compute than $\vol(\conv(\cA))$ and (\ref{Equation: Khovanskii lead term}) depends on the difference set $\cA - \cA$. \fi 
Let us see some examples where we compute the Hilbert function and polynomial and, hence, the degree of $Y_{n,d_{\cA}}$.

\begin{example}\label{Example: examples section 3} \rm (i) We take integers $n,d \geq 1$ and $\cA = \{(a_1,\hdots,a_n) \in \ZZ_{\geq 0}^{n} \mid a_1+\cdots+a_n\leq d\}.$ Then, $\overline{\cA} = \{(a_0,\hdots,a_n) \in \ZZ_{\geq 0}^{n+1} \mid a_0+\cdots +a_n = d\}$ and the associated monomial projection is the Veronese variety $X_{n,d}$. It is straightforward to see that $\rk(M) = n+1$ and $\vol(\conv(\overline{\cA})) = \frac{d^{n+1}}{(n+1)!}$. To compute $\Delta_{n+1}$ in this case, it is enough to find the {\em Smith normal form} of $M$ (see \cite{Carroll-Planas-Villarreal}). Notice that the transpose matrix $M^t$ of $M$ contains $n$ rows corresponding to $f_i = (d-1,\hdots, 1,\hdots, 0)$ with $1$ in position $i$th, $i = 1,\hdots,n$. Consider the submatrix $M_{a}$ of $M^t$ whose rows are $a,f_1,\hdots,f_n$ with $a = (a_0,\hdots,a_n) \in \overline{\cA} \setminus \{f_1,\hdots,f_n\}$.  Then, by doing the elementary row operation $a - a_1f_1 - \cdots - a_nf_n$, we can transform the row  $a$ as $(|a| - d(a_1+\cdots+a_n),0,\hdots,0)$. Since $|a|$ is a multiple of $d$ and, in particular, $(d,0,\hdots,0) \in \overline{\cA}$, we obtain that the Smith normal form of $M$ is $\diag(1,\hdots,1,d,0,\hdots,0)$. Therefore, $\Delta_{n+1} = d$ and we get $\deg(X_{n,d}) = d^{n}$. On the other hand, for $t \geq 0$ we have
\[\varphi_{\cA}(t) = p_{\cA}(t) = \binom{n+dt}{n},\]
so $n_0(\cA) = 0$ and the leading coefficient of $p_{\cA}(t)$ is $\frac{d^{n}}{n!}$.

\noindent (ii) We take $n = 2$ and $\cA = \{(0,0), (1,0), (0,1), (2,0), (0,2), (3,0), (2,1), (1,2), (0,3)\}$. Then, $|\cA| = 9$, $d_{\cA} = 3$ and the associated monomial projection $Y_{2,3} \subset \PP^{8}$ of the Veronese surface $X_{2,3} \subset \PP^{9}$ is the surface parameterized by 
$\Omega_{2,3} = \{x_0^3, x_0^2x_1, x_0^2x_2, x_0x_1^2, x_0x_2^2, x_1^3, x_1^2x_2, x_1x_2^2, x_2^3\}$.
We have $\overline{\cA} = \{(3,0,0), (2,1,0), (2,0,1), (1,2,0), (1,0,2), (0,3,0), (0,2,1), (0,1,2), (0,0,3)\}$. It is straightforward to check that 
$\vol(\conv(\overline{\cA})) = 27$ and, computing the maximal minors of $M$, we get $\Delta_{3} = 3$. By Proposition \ref{Equation: degree from lattice ideal}, $\deg(Y_{2,3}) = 9$. On the other hand, for $t \geq 2$ we have
\[\varphi_{\cA}(t) = p_{\cA}(t) = \frac{9t^2+9t+2}{2} = \binom{3t+2}{2},\]
so $n_0(\cA) = 2$ and the leading coefficient of $p_{\cA}(t)$ is $ \frac{9}{2}$.

\noindent (iii) We take $n = 2$ and $\cA = \{(0,0), (1,1), (3,0), (0,3)\}$. Then, $|\cA| = 4$, $d_{\cA} = 3$ and the associated monomial projection $Y_{2,3} \subset \PP^{3}$ of the Veronese surface $X_{2,3} \subset \PP^{9}$ is the surface parameterized by 
$\Omega_{2,3} = \{x_0^3, x_0x_1x_2, x_1^3, x_2^3\}$.
We have $\overline{\cA} = \{(3,0,0), (1,1,1), (0,3,0), (0,0,3)\}$. It is straightforward to check that 
$\vol(\conv(\cA)) = 27$ and, computing the maximal minors of $M$, we have $\Delta_{3} = 9$. By Proposition \ref{Equation: degree from lattice ideal}, $\deg(Y_{2,3}) = 3$.  On the other hand, for $t \geq 0$ we have 
\[\varphi_{\cA}(t) = p_{\cA}(t) = \frac{3t^2+3t+2}{2},\]
so $n_0(\cA) = 0$ and the leading coefficient of $p_{\cA}(t)$ is $\frac{3}{2}$.

\noindent (iv) We take $n = 2$ and $\cA = \{(2,2), (2,0), (1,2), (0,4)\}$. Then, $|\cA| = 4$, $d_{\cA} = 4$ and the associated monomial projection $Y_{2,4} \subset \PP^{3}$ of the Veronese surface $X_{2,4} \subset \PP^{9}$ is the surface parameterized by 
$\Omega_{2,4} = \{x_1^2x_2^2, x_0^2x_1^2, x_0x_1x_2^2, x_2^4\}$. We have $\overline{\cA} = \{(0,2,2), (2,2,0), (1,1,2), (0,0,4)\}$,  $\rk(M) = 3$,  $\vol(\conv(\overline{\cA})) = \frac{8}{3}$ and, computing the maximal minors of $M$, we get $\Delta_{3} = 8$. By Proposition \ref{Equation: degree from lattice ideal}, $\deg(Y_{2,4}) = \frac{3!\cdot 8}{3\cdot 8} = 2$. On the other hand, for $t \geq 0$ we have
\[\varphi_{\cA}(t) = p_{\cA}(t) = t^2+2t+1,\]
so $n_0(\cA) = 0$ and the leading coefficient of $p_{\cA}(t)$ is $1$. 
\end{example}

We end this section with a purely geometric approach to calculate the leading term of the polynomial $p_{\cA}(t)\in \Q[t]$ associated to any finite subset $\cA\subset \ZZ_{\ge 0}^n$. The result is based on the following observation:
we consider two finite subsets $\cA_1\subset\cA_2\subset \ZZ^n_{\ge 0}$ with $|\cA_2|=|\cA_1|+1$ and associated $n$-dimensional projective varieties $Y_{n,d_{\cA _1}}\subset \PP^{|\cA _1|-1}$ and $Y_{n,d_{\cA _2}}\subset \PP^{|\cA _2|-1}\cong \PP^{|\cA _1|}$. Notice that 
$Y_{n,d_{\cA _1}}$ is obtained projecting $Y_{n,d_{\cA _2}}$ from a point $p_{2,1}\in \PP^{|\cA _2|-1}$. Denote by $\pi _{2,1}: Y_{n,d_{\cA _2}} \longrightarrow Y_{n,d_{\cA _1}}$ the projection; it is a finite morphism of degree $\deg \pi _{2,1}= \# \pi _{2,1}^{-1}(x)$ where $x\in Y_{n,d_{\cA _1}}$ is a general point.
 By \cite[Pgs. 234-235, 259]{Ha}, we have:
% \vskip 2mm
\begin{equation}
\label{degree}
\deg \pi _{2,1} \cdot \deg Y_{n,d_{\cA _2}}=
\begin{cases} \deg Y_{n,d_{\cA _2}} &  \text{ if }   p_{2,1}\notin Y_{n,d_{\cA _1}} \\
\deg Y_{n,d_{\cA _2}}-1 &  \text{ if }   p_{2,1}\in Y_{n,d_{\cA _1}} \text{ is a smooth  point} \\
\deg Y_{n,d_{\cA _2}}-m_{2,1} &  \text{ if  }  p_{2,1}\in Y_{n,d_{\cA _1}} \text{ is a  point of multiplicity } m_{2,1}.
\end{cases}
\end{equation}

\noindent
Iterating this process we can compute the leading term of the polynomial $p_{\cA}(t)\in \Q[t]$ associated to any finite subset $\cA\subset \ZZ_{\ge 0}^n$. 
To this end, we need to fix some extra notation. Given any finite subset $\cA\subset \ZZ^n_{\ge 0}$ with $r:=|\cA|$,  we consider a chain
$$\cA=\cA_{0}\subset \cA_1\subset \cA_2\subset \cdots \subset  \cA_{N_{n,d}-r}=\cM_{n,d}\subset \ZZ^n_{\ge 0}$$
where $|\cA_{i}|=|\cA_{i-1}|+1$ and $\cM_{n,d}$ denotes the set of all monomials of degree $d$ in $R$. The $n$-dimensional rational projective variety $Y_{n,d_{\cA _{i-1}}}\subset \PP^{|\cA_{i-1}|-1}$ is obtained projecting $Y_{n,d_{\cA _{i}}}\subset \PP^{|\cA_{i}|-1}$ from a point $p_{i,i-1}$. Call $\pi _{i,i-1}$ such a projection. We define

$$
d_{i,i-1}= \begin{cases}
0 & \text{ if } p_{i,i-1}\notin Y_{n,d_{\cA _{i}}} \\
1 & \text{ if } p_{i,i-1}\notin Y_{n,d_{\cA _{i}}} \text{ is a smooth point} \\
m_{i,i-1} & \text{ if } p_{i,i-1}\notin Y_{n,d_{\cA _{i}}} \text{ is a point of multiplicity }m_{i,i-1}.
\end{cases}$$

\begin{proposition}\label{Proposition: geometric interpretation degree} Let $p_{\cA}(t)\in \Q[t]$ be the polynomial of a finite subset $\cA\subset \ZZ_{\ge 0}^n$  with associated $n$-dimensional projective $Y_{n,d_{\cA }}$ and let $a_n$ be its leading coefficient. With the above notation it holds:
$$
a_n= \frac{\deg Y_{n,d_{\cA }}}{n!}=\frac{1}{n! \prod _{i=1}^{N_{n,d}-r}\deg \pi _{i,i-1}} 
  \big (d^n-
\sum_{i=1}^{N_{n,d}-r}(m_{i,i-1}\prod _{j=i+1}^{N_{n,d}-r} \deg \pi _{j,j-1})\big ).$$

\end{proposition}

\begin{proof}
It immediately follows from (\ref{degree}) taking into account that $\deg X_{n,d}=d^n$.
\end{proof}

Next example illustrates the above results.
\begin{example} \rm
We consider the finite set $\cA =\{ (0,0),(3,0),(2,0),(2,2),(0,1)\}\subset \ZZ^n_{\ge 0}$ and the associated rational projective surface $Y_{2,4}\subset \PP^4$, i.e.  the surface parameterized by $\Omega_{2,4} = \{x_0^4,x_0x_1^3,x_0^2x_1^2,x_1^2x_2^2,x_0^3x_2\}$. If we fix coordinates $w_0, \hdots, w_4$ in $\PP^4$, the ideal $\I(Y_{2,4})\subset S=\kk[w_0, \hdots, w_4]$ is generated by $w_2^3-w_0w_3^2,w_1^2w_2-w_0^2w_4,w_1^2w_3^2-w_0w_2^2w_4$ and it has a minimal graded free $S$-resolution:
$$
0\longrightarrow S(-5)^2 \xrightarrow{d_2} S(-3)^2 \oplus  S(-4)\xrightarrow{d_1} \I(Y_{2,4})\longrightarrow 0,
$$
where the graded $S-$maps $d_1$ and $d_2$ are associated to the matrices
\[\left(\begin{array}{lll}
    w_2^3-w_0w_3^2 &  w_1^2w_2-w_0^2w_4 & w_1^2w_3^2-w_0w_2^2w_4\\
\end{array}\right) \quad \text{and} \quad \left(\begin{array}{cc}
 -w_0w_3 & w_4^2\\
 w_1^2  & -w_2^2\\
 -w_2   & w_0\\
 \end{array}\right),\]
respectively. Therefore, we have
$$
\HP_{Y_{2,4}}(t)=\varphi _{\cA}(t)=|t\cA|= 4t^2-2t+3 \text{ for all } t\ge 1.
$$
In particular, $Y_{2,4}$ is a degree $8$ surface in $\PP^{4}$. The Example \ref{ex2} (iv) can be recovered from suitable projections of  $Y_{2,4}$. Indeed, we fix the points $p_0=(1:0:0:0:0), \hdots , p_4=(0:0:0:0:1)$ and we denote by $\pi _{i}:Y_{2,4} \longrightarrow \PP^3$ the projection of $Y_{2,4}$ to $\PP^3$ from the point $p_{i}$, $i=0,\cdots ,4$. It holds $\pi _{2}(Y_{2,4})=Y_{2,4}^1$ and $\pi _{3}(Y_{2,4})=Y_{2,4}^2$. By the above result, since $p_2\notin Y_{2,4}$, $p_3\in Y_{2,4}$ is a double point, $\deg \pi _2=1$ and $\deg \pi _3=2$,  we obtain: $$\deg(Y_{2,4})=\deg(Y_{2,4}^1)=8 \text{ and } \deg Y_{2,4}^2=\frac{\deg (Y_{2,4})-2}{2}=3.$$
The result fits well with our previous calculations (see Examples \ref{firstexample} and \ref{ex2}) and we have:
$$
\HP_{Y_{2,4}^1}(t)=\varphi _{\cA^1}(t)=|t\cA ^1|=4t^2-16t+36 \text{ for all } t\ge 5
$$
and
$$
\HP_{Y_{2,4}^2}(t)=\varphi _{\cA ^2}(t)=|t\cA ^2|=\frac{3}{2}(t^2+t)+1  \text{ for all } t\ge 0.
$$
being $\cA^1=\{ (0,0),(3,0),(2,2),(0,1)\}, \cA^2=\{ (0,0),(2,0),(2,2),(0,1)\}\subset \ZZ^n_{\ge 0}$ the finite sets associated to $Y_{2,4}^1$ and $Y_{2,4}^2$, respectively.

We end the example with one more computation to show the casuistry we can have. Since $\pi _0:Y_{2,4} \longrightarrow \PP^3$ is a finite morphism of degree 1 and $p_0\in Y_{2,4}$ a double point, we get that $Y_{2,4}^0:=\pi _{0}(Y_{2,4})\subset \PP^3$ is a surface of degree 6 parameterized by  $\{x_0x_1^3, x_0^2x_1^2, x_1^2x_2^2, x_0^3x_2\}$ and the associated finite set $\cA ^0=\{(3,0),(2,0),(2,2),(0,1)\}\subset \ZZ^n_{\ge 0}$ satisfies:
$$
\HP_{Y_{2,4}^0}(t)=\varphi _{\cA ^0}(t)=|t\cA ^0|=3t^2-6t+11  \text{ for all } t\ge 3.
$$
\end{example}

%%%%%%%%%%%%%%%%%%%%%%%%%%%%%%%%%%%%%%%%%%%%
%%%%%%%%%%%%%%%%%%%%%%%%%%%%%%%%%%%%%%%%%%%%
%%%%%%%%%%%%%%%%%%%%%%%%%%%%%%%%%%%%%%%%%%%%
%%%%%%%%%%%%%%%%%%%%%%%%%%%%%%%%%%%%%%%%%%%%
%%%%%%%%%%%%%%%%%%%%%%%%%%%%%%%%%%%%%%%%%%%%
%%%%%%%%%%%%%%%%%%%%%%%%%%%%%%%%%%%%%%%%%%
%%%%%%%%%%%%%%%%%%%%%%%%%%%%%%%%%%%%%%%%%%

\section{Bounds for the phase transition}
\label{sec-phase}
%Thus far we have associated to any finite subset $\cA \subset \ZZ^n_{\ge 0}$, on one hand, a function %$\varphi_{\cA}(t) = |t\cA|$ and polynomial $p_{\cA }(t)\in \Q[t]$ such that $p_{\cA}(t)=\varphi_{A}(t)$ for all %$t\gg 0$ and, on the other hand, a monomial projection $Y_{n,\cA}$ of the Veronese variety $X_{n,d} \subset %\PP^{N_{n,d}-1}$ whose Hilbert function and polynomial coincides with $\varphi_{\cA}(t)$ and $p_{\cA}(t)$, %respectively.  
In Section \ref{Section: Additive number theory and projective geometry},  we have studied the leading coefficient of $p_{\cA}(t)$ and we have provided combinatorial and geometric interpretations of it. 
Notwithstanding, for arbitrary $n \geq 1$ and $\cA \subset \ZZ^{n}$, the phase transition $n_{0}(\cA)$ and the polynomial $p_{\cA}(t)$ are barely known. We will use the Castelnuovo--Mumford regularity $\reg(\cA)$ of $A(Y_{n,d_{\cA}})$ to derive new bounds for $n_0(\cA)$ under some technical assumptions on  $\cA$. 

Following \cite{Granville-Shakan-Walker} and \cite{Curran-Goldmakher}, we gather a series of known bounds for the phase transition in the following setting: 
 finite subsets $\cA \subset \ZZ^{n}_{\ge 0}$  whose convex hull $\conv(\cA)$  is an $n-$simplex. In this case, the bounds on  the phase transition $n_{0}(\cA)$ are expressed in terms of $n,|\cA|, \vol(\conv(\cA))$ and a constant $K(\cA,B)$, which depends further on $\cA$ and its associated semigroup $\cH(\cA) \subset \ZZ^n_{\ge 0}$.  

\begin{definition}\label{Definition: A with n-simplex}Let $\cA \subset \ZZ^{n}_{\ge 0}$ be a finite subset. The convex hull $\conv(\cA)$ of $\cA$ is an $n-$simplex if there is a subset $B = \{v_1,\hdots,v_{n+1}\}\subset \cA$ of $n+1$ elements such that the difference set $B-B$ generates $\RR^{n}$ and $\conv(\cA) = \conv(B)$. 
\end{definition}

\begin{theorem}\label{Theorem: effectives khovanskii} Let $\cA \subset \ZZ^{n}_{\ge 0}$ be a finite subset with $\conv(\cA)$ an $n-$simplex. We  have:
\[ 
n_0(\cA) \leq (n+1)(n!\frac{\vol(\conv(\cA))}{[\ZZ^n:\ZZ(\cA-\cA)]} - |\cA|+n) + 1.
\]
If in addition $\ZZ(\cA - \cA) = \ZZ^n$, then  
\[
n_0(\cA) \leq (n+1)!\vol(\conv(\cA))-\max\{ 3n+1, (n+1)(|\cA|-n) - 1\}.
\]
\end{theorem}
\begin{proof}
See \cite[Theorem 1.4]{Granville-Shakan-Walker} and \cite[Theorem 1.4]{Curran-Goldmakher}. 
\end{proof}

Assume now that $\cA \subset \ZZ^{n}_{\ge 0}$ is a finite subset such that $0 \in \cA$ and $\conv(\cA)$ is an $n-$simplex. We denote by $\ZZ(B) \subset \ZZ^{n}$ the subgroup generated by $B$ and we set $\Pi_{B} := \{\sum_{i=1}^{n+1} \lambda_iv_i \mid 0 \leq \lambda_i < 1\}$. Given $a \in \cH(\cA) = \cup_{t \geq 0}(t\cA)$, the {\em height} of $a$ is the minimum of the integers $t \geq 0$ such that $a \in t\cA$, we denote it by $N_{a}$. The class of $a$ modulo $\ZZ(B)$ can be represented by an element of $\pi_{a} \in \Pi_{B}$ and we denote by $S_{\pi}$ the set of elements of $\cH(\cA)$ which are congruent to $\pi_{a}$ modulo $\ZZ(B)$. An element $a \in S_{\pi}$ is said to be {\em $B-$minimal} if $a - v_{i} \notin \cH(\cA)$ for any $i = 1,\hdots,n+1$. The set of all $B-$minimal elements of $\cH(\cA)$ is denoted by $\cS(\cA,B)$. Keeping this notation:

\begin{proposition}{\cite[Theorem 4.1]{Granville-Shakan-Walker}} \label{Proposition: boudn K(A,B)}  If $\cS(\cA,B)$ is finite, then
\[
n_0(\cA) \leq (n+1)(K(\cA,B)-1)+1, 
\] 
where $\displaystyle{K(\cA,B) := \max_{a \in \cS(\cA,B)} N_{a}}$. 
\end{proposition}

In next example, we compute the bounds given in Theorem \ref{Theorem: effectives khovanskii}  and Proposition \ref{Proposition: boudn K(A,B)} and we show that both bounds are far from the real value of $n_0(\cA)$.

\begin{example}\label{Example: example section 4} \rm Let $2 \leq n < d$ and $0 \leq \alpha_1 \leq \cdots \leq \alpha_n < d$ be integers such that $\GCD(\alpha_1,\hdots, \alpha_n,d) = 1$. Set 
\[\cA = \{(a_1,\hdots, a_n) \in \ZZ_{\geq 0}^{n} \mid a_1 + \cdots + a_n \leq d \; \text{and} \; \alpha_1a_1 + \cdots + \alpha_na_n \equiv 0 \mod d\}.\]
We have that $d_{\cA} = d$ and $\cA$ contains $0$ and the elements $de_i := (0,\hdots, d, \hdots, 0)$ with $d$ in position $i$th, $i = 1,\hdots, n$. Then, $\cA$ fulfils the hypothesis of Proposition \ref{Proposition: boudn K(A,B)} with $B = \{0,de_1,\hdots,de_n\}$ and $\vol(\conv(\cA)) = \frac{d^{n}}{n!}$. Moreover, in \cite{CMM-R} it is proved that the set $\overline{\cA}$ generates the semigroup 
\[H = \{(a_0,\hdots, a_n) \in \ZZ^{n+1}_{\geq 0} \mid a_0 + \cdots + a_n \equiv 0\! \mod d \; \text{and} \; \alpha_1a_1 + \cdots + \alpha_na_n \equiv 0 \mod d\}\]
Take $a \in \cH(\cA)$. If $a = (a_1,\hdots, a_n) - de_i \in \ZZ_{\geq 0}^{n}$, then $a-de_i \in \cH(\cA)$. Therefore, $\cS(\cA,B) = \{a = (a_1,\hdots, a_n) \in \cA \mid a_i < d, i = 1,\hdots, n\}$ and, hence, $K(\cA,B) \leq n$. Applying Proposition \ref{Proposition: boudn K(A,B)}, we obtain that $n_0(\cA) \leq (n+1)(n-1) + 1$. As a particular example, we take $n = 2$, $d = 5$, $\alpha_1 = 1$ and $\alpha_2 = 2$. Then, we have 
\[\cA = \{(0,0), (5,0), (3,1), (1,2), (0,5)\},\]
$d_{\cA} = 5$, $|\cA| = 5$, $B = \{(0,0), (5,0), (0,5)\}$, $\vol(\conv(\cA)) = \frac{25}{2}$ and $K(\cA,B) = 2$. Notice that the subgroup $\ZZ(\cA-\cA)$ does not coincide with $\ZZ^{2}$. The bound for the phase transition of $\varphi_{\cA}(t)$ from Theorem \ref{Theorem: effectives khovanskii} is $n_{0}(\cA) \leq 5$; in contrast to the lower one $n_0(\cA) \leq 4$ from Proposition \ref{Proposition: boudn K(A,B)}. Notwithstanding, we have that 
\[\varphi_{\cA}(t) = p_{\cA}(t) = \frac{5t^2+3t+2}{2} \; \text{for all} \; t \geq 0\]
(\cite[Theorem 4.12]{CMM-R}), so the phase transition of $\cA$ is $n_0(\cA) = 0$. 
\end{example}

Next, we establish new bounds for the phase transition using the geometric approach we have developed in Section \ref{Section: Additive number theory and projective geometry}.  
The first bounds we provide for $n_0(\cA)$ are expressed in terms of $n, |\cA|$ and  $\deg(Y_{n,d_{\cA}})$, and they can be easily compared with the previous ones using (\ref{Equation: Khovanskii lead term}). The last bound we give is based on the reduction number $r(\cA)$ of $\kk[Y_{n,d_{\cA}}]$ which, as the constant $K(\cA,B)$, depends further on the subset $\overline{\cA}$. 

Through the rest of this section, we consider subsets $\cA \subset \ZZ_{\geq 0}^{n}$ such that $\overline{\cA} \subset \ZZ^{n+1}_{\geq 0}$ generates a {\em simplicial} affine semigroup $\cH(\overline{\cA})$, i.e. $e_0 := (d_{\cA}, 0, \hdots, 0), \hdots, e_n := (0,\hdots,0, d_{\cA}) \in \overline{\cA} \subset \ZZ_{\geq 0}^{n+1}$. The homogeneous coordinate ring $A(Y_{n,d})$ of the monomial projection $Y_{n,d}$ associated to $\cA$ is isomorphic to $\kk[\Omega_{n,d_{\cA}}]$,  which is the simplicial semigroup ring associated to the affine semigrop $\cH(\overline{\cA}) \subset \ZZ_{\geq 0}^{n+1}$. This is equivalent to consider subsets $\cA \subset \ZZ_{\geq 0}^{n}$ containing the origin $0 \in \ZZ^{n}$ and the elements, which we fix in the sequel, $v_1 := (d_{\cA},0,\hdots,0), \hdots, v_n := (0,\hdots,0,d_{\cA}) \in \cA \subset \ZZ_{\geq 0}^{n}$. The convex hull $\conv(\cA)$ is an $n-$simplex with vertexes $0, v_1,\hdots, v_n$ and $\vol(\conv(\cA)) = \frac{d_{\cA}^{n}}{n!}$. In particular, $\ZZ(\cA-\cA)$ has maximal rank, hence ((\ref{Equation: Khovanskii lead term}) and (\ref{Equation: degree from lattice ideal})):  
\[\deg(Y_{n,d_{\cA}}) = n!\frac{\vol(\conv(\cA))}{[\ZZ^n : \ZZ(\cA-\cA)]} = 
(n+1)!\frac{\vol(\conv(\overline{\cA}))}{\Delta_{n+1}} = \frac{d_{\cA}^n}{[\ZZ^n : \ZZ(\cA-\cA)]}.\]  We have:
 
\begin{theorem}\label{Theorem: Bounds reg Hoa-Stuckrad} Let $\cA \subset \ZZ_{\geq 0}^{n}$ be a finite subset such that $\{0,v_1,\hdots,v_n\} \subset \cA$. If one of the following conditions yields,
   \begin{itemize}
       \item[(i)] $n = 1$,
       \item[(ii)]  $\kk[\Omega_{n,d_{\cA}}]$ is a Cohen-Macaulay ring,
       \item[(iii)] $\deg(Y_{n,d_{\cA}}) \leq |\cA| - n$, 
       \item[(iv)] $|\cA| - n - 1 \leq \deg(Y_{n,d_{\cA}})/d_{\cA}$ or
       \item[(v)] $\deg(Y_{n,d_{\cA}}) = d_{\cA}^{n}$ and $d_{\cA} \leq n$, 
   \end{itemize} 
then 
\[n_0(\cA)  \leq \deg(Y_{n,d_{\cA}}) - |\cA| + n + 2.\]
If $Y_{n,d_{\cA}}$ is a smooth variety, then we have further
\[n_0(\cA)  \leq \min\{n(d_{\cA}-2)+1, \deg(Y_{n,d_{\cA}}) - |\cA| + n + 2\}.\]
\end{theorem}
\begin{proof} Under the hypothesis of the statement, the Eisenbud-Goto conjecture for the Castelnuovo-Mumford regularity $\reg(\cA)$ of $\kk[\Omega_{n,d_{\cA}}]$ holds
(see \cite[Corollary 3.6 and Proposition 3.7]{Hoa-Stuckrad} and \cite[Theorem 1.1 and Corollary 1.3]{Herzog-Hibi}). Therefore, the proof now follows from the general fact:
$$n_0(\cA)\le \reg(\cA)+1.
$$
\end{proof}

In general, 
\begin{theorem}\label{Theorem: Bounds reg Hoa-Stuckrad-continuacio}  Let $\cA \subset \ZZ_{\geq 0}^{n}$ be a finite subset such that $\{0,v_1,\hdots,v_n\} \subset \cA$. Then, 
\[n_0(\cA) \leq (d_{\cA}-1)(|\cA|-n-1)+1.\]
Moreover, if $\deg(Y_{n,d_{\cA}}) \geq |\cA| - n + 1$, then
    \[n_0(\cA) \leq \min\{(n+1)(\deg(Y_{n,d_{\cA}})-|\cA|+n-1)+3, (d_{\cA}-1)(|\cA|-n-1)+1\}.\]
\end{theorem}
\begin{proof} The result follows from \cite[Theorem 3.2 and Theorem 3.5]{Hoa-Stuckrad}. 
\end{proof}

Overall, combining Theorems \ref{Theorem: effectives khovanskii}, \ref{Theorem: Bounds reg Hoa-Stuckrad} and \ref{Theorem: Bounds reg Hoa-Stuckrad-continuacio} we obtain finer bounds for the phase transition $n_0(\cA)$.  For instance, subsets $\cA \subset \ZZ_{\geq 0}^{n}$ containing $\{0,v_1,\hdots,v_n\}$ and moreover 
any point of the form $(0,\hdots,\alpha (d_{\cA}-1), \hdots, \beta, \hdots,0)$ with $\alpha,\beta \in \{0,1\}$ give rise to a smooth monomial projection $Y_{n,d_{\cA}}$ of the Veronese variety $X_{n,d_{\cA}}$. 
In this case, $\ZZ(\cA - \cA) = \ZZ^{n}$ and we have $\deg(Y_{n,d_{\cA}}) = d_{\cA}^n$. 
Therefore, 
\[n_0(\cA) \leq \min\{d_{\cA}^{n} + n(d_{\cA}^{n} - 3) - 1, n(d_{\cA}-2)+1, d_{\cA}^{n} - |\cA| + n + 2\},\]
which gives $\min\{n(d_{\cA}-2)+1, d_{\cA}^{n} - |\cA| + n + 2\}$ whenever $d_{\cA}^{n} - n - 1 \geq 1$ or $n(d_{\cA}^{n}-3)-1 > n+2-|\cA|$.  
\begin{remark}[Comparision of bounds]\rm Assume that $\cA \subset \ZZ_{\geq 0}^{n}$ is a finite subset with $n \geq 2$ and $\{0,v_1,\hdots,v_n\} \subset \cA$. Without further assumptions, by Theorem \ref{Theorem: Bounds reg Hoa-Stuckrad}(iii) and Theorem \ref{Theorem: Bounds reg Hoa-Stuckrad-continuacio}, we have that $n_0(\cA) \leq (n+1)(\deg(Y_{n,_{\cA}}) -|\cA| + n - 1)+3$ which improves by $n-1$ the bound $n_0(\cA) \leq (n+1)(\deg(Y_{n,\cA}) - |\cA| + n) + 1$ given in \cite{Granville-Shakan-Walker} and the bound $(n+1)\deg(Y_{n,\cA}) - 3n-1$ given in \cite{Curran-Goldmakher} when $|\cA| > n+2$. On the other hand, for those subsets $\cA$ satisfying one of the hypothesis of Theorem \ref{Theorem: Bounds reg Hoa-Stuckrad}, we have 
$n_0(\cA) \leq \deg(Y_{n,d_{\cA}}) - |\cA| + n + 2$ which beats the prior bounds in almost all cases and it is close to the bound $n_0(\cA) \leq \deg(Y_{n,\cA}) - |\cA| + 2$ conjectured in \cite{Curran-Goldmakher}. In particular, when $Y_{n,\cA}$ is a smooth variety we have $n_0(\cA) \leq \min\{\deg(Y_{n,\cA}) - |\cA| +n - 1, n(d_{\cA} - 2) + 1\}$. The last expressions are more difficult to compare in general. Notwithstanding, for this kind of monomial projections, it is often the case $\deg(Y_{n,\cA}) = d_{\cA}^{n}$, or equivalently $[\ZZ^n:\ZZ(\cA-\cA)] = 1$. For instance, as we have seen before, $Y_{n,d_{\cA}}$ is a degree $d_{\cA}^n$ smooth variety when $\cA$ contains any point of the form $(0,\hdots, \alpha(d_{\cA}-1), \hdots, \beta, \hdots,0)$ with $\alpha,\beta \in \{0,1\}$. 
In this setting, we have $n(d_{\cA} - 2) + 1 \leq d_{\cA}^n - |\cA| + n - 1$. Indeed, $|\cA| \leq \binom{n+d_{\cA}}{n}$ and the inequality $n(d_{\cA} - 2)+1 \leq d_{\cA}^n - \binom{n+d_{\cA}}{n} + n-1$ holds. 
\end{remark}
In \cite{Hoa-Stuckrad}, the authors provide bounds for $\reg(\cA)$ in terms of the reduction number $r(\cA)$ of $\kk[\Omega_{n,d_{\cA}}]$. We will relate $r(\cA)$ to the constant $K(\cA,B)$, with $B = \{0,v_1,\hdots,v_n\}$, and we will provide bounds for $n_{0}(\cA)$ in terms of $r(\cA)$ that improve the one given in Proposition \ref{Proposition: boudn K(A,B)}. 

\begin{definition} Let $\cA \subset \ZZ_{\geq 0}^{n}$ be a finite subset such that $\{0,v_1,\hdots,v_n\} \subset \cA$. The reduction number $r(\cA)$ of the semigroup ring $\kk[\Omega_{n,d_{\cA}}]$ is the least positive integer $r$ such that 
$(r+1)\overline{\cA} = \{e_0,\hdots, e_n\} + r\overline{\cA}.$
\end{definition}

\begin{remark}\rm \label{Remark: reduction number} Notice that for any $r \geq  r(\cA)$, we obtain inductively that  
\[(r+1)\overline{\cA} = \{e_0,\hdots, e_n\} + r\overline{\cA}.\]
\end{remark}

 In the next result we relate $r(\cA)$ to $K(\cA,B)$:

\begin{proposition} \label{Lemma: r(A) <= K(A,B)} Let $\cA \subset \ZZ_{\geq 0}^{n}$ be a finite subset such that $\{0,v_1,\hdots,v_n\} \subset \cA$. Then, $r(\cA) = K(\cA,B)$. 
\end{proposition}

\begin{proof} To prove the result, we check the inequalities $r(\cA) \leq K(\cA,B)$ and $K(\cA,B) \leq r(\cA)$. Let $r \geq K(\cA,B)$ be an integer and $a = (a_1, \hdots, a_n) \in (r+1)\cA$.  We set $a_0 = (r+1)d_{\cA}-a_1-\cdots-a_n$  and we denote $\overline{a} = (a_0,a_1,\hdots,a_n) \in (r+1)\overline{\cA}$. By hypothesis, there exists $v_i$ such that $a-v_i \in r\cA$. Hence, $\overline{a} \in \{e_0+\cdots + e_n\} + r\overline{\cA}$. So, $r(\cA) \leq K(\cA,B)$.

Now, let $r \geq r(\cA)$ be an integer and assume that $a \in \cS(\cA,B)$ has height $N_a = r+1$, i.e. $a \in (r+1)\cA$ is $B-$minimal and $a \notin r'\cA$ for any $r' \leq r+1$. By hypothesis,  $(r+1)\overline{\cA} = \{e_0,\hdots,e_n\} + r\overline{\cA}$, so $\overline{a}-e_i \in r\overline{\cA}$ for some $0 \leq i \leq n$. Since $a$ is $B-$minimal, it follows that $\overline{a}-e_i \notin r\overline{\cA}$ for any $1 \leq i \leq n$. Thus, $\overline{a}-e_0 \in r\overline{\cA}$. In particular, we obtain that $a \in r\cA$ which contradicts the hypothesis $N_a = r+1$. As a result, any $B-$minimal element has height $N_a \leq r(\cA)$, so $K(\cA,B) \leq r(\cA)$.
\end{proof}

We have: 
\begin{theorem}\label{Theorem: Bounds reg Hoa-Stuckrad-reduction number} Let $\cA \subset \ZZ_{\geq 0}^{n}$ be a finite subset such that $\{0,v_1,\hdots,v_n\} \subset \cA$. We have:
\begin{itemize}
    \item[(i)] if $r(\cA) \leq 1$, then $n_{0}(\cA) \leq 2$, and
    \item[(ii)] if $r(\cA) > 1$, then 
\[n_0(\cA) \leq \min\{(n+1)(r(\cA)-1)-n+2, (n+1)r(\cA) - \lceil \frac{(n+1)r(\cA)}{d_{\cA}}\rceil + 1\}.\]
\end{itemize}
\end{theorem}
\begin{proof} By Proposition \ref{Lemma: r(A) <= K(A,B)} we have  $r(\cA) = K(\cA,B)$. Hence, applying 
\cite[Theorems 3.1 and 3.2]{Hoa-Stuckrad} we  get what we want.
\end{proof}

As we have established before in Proposition \ref{Lemma: r(A) <= K(A,B)}, when $\cA \subset \ZZ_{\geq 0}^{n}$ contains $\{0,v_1,\hdots, v_n\}$, we have the equality of the constants $K(\cA,B) = r(\cA)$. In this setting, it is natural to compare the bounds for $n_0(\cA)$ which are expressed in terms of them.  By Theorem  \ref{Theorem: Bounds reg Hoa-Stuckrad-reduction number} and Proposition \ref{Lemma: r(A) <= K(A,B)}, it holds that $n_0(\cA) \leq (n+1)(K(\cA,B)-1)-n+2$. Now by Proposition \ref{Proposition: boudn K(A,B)}, $n_0(\cA) \leq (n+1)(K(\cA,B)-1)+1$. 
Accordingly, Theorem \ref{Theorem: Bounds reg Hoa-Stuckrad-reduction number} improves at least by $n-1$, which is always positive when $n \geq 2$, the bound for the phase transition $n_0(\cA)$ given in Proposition \ref{Proposition: boudn K(A,B)}. Let us see a more concrete example. 

\begin{example}\label{Example: Section 4 compare} \rm  Continuing with Example \ref{Example: example section 4}, we take $\cA = \{(0,0), (5,0), (3,1), (1,2), (0,5)\} \subset \ZZ_{\geq 0}^{2}$ and we have $d_{\cA} = 5$, $|\cA| = 5$, $\deg(Y_{2,5}) = 5$, $B = \{(0,0), (5,0), (0,5)\}$, $\vol(\conv(\cA)) = \frac{25}{2}$, $K(\cA,B) = r(\cA) = 2$. As we have seen before, the bounds from Theorem \ref{Theorem: effectives khovanskii} and Proposition \ref{Proposition: boudn K(A,B)} are $n_0(\cA) \leq 5$ and $n_{0}(\cA) \leq 4$, respectively. Thus far, we can assure that $n_0(\cA) \leq 4$. Since $\kk[\Omega_{2,5}]$ is a Cohen--Macaulay ring (\cite[Theorem 3.3]{CMM-R}), by Theorem \ref{Theorem: Bounds reg Hoa-Stuckrad} we obtain $n_0(\cA) \leq 4$; and by Proposition \ref{Theorem: Bounds reg Hoa-Stuckrad-reduction number} we get $n_0(\cA) \leq 3$, which overall improves the previous bound. 
\end{example}

To determine $K(\cA,B)$, or equivalently $r(\cA)$, could be cumbersome depending on the subset $\cA$. In \cite{Granville-Shakan-Walker}, the authors provide a bound for $K(\cA,B)$ in terms of the {\em Davenport constant} of a certain group. On the other hand, for subsets $\cA \subset \ZZ_{\geq 0}^{n}$ with $\{0,v_1,\hdots,v_n\} \subset \cA$, the constant $r(\cA)$ can be also bounded as follows:

 \begin{proposition}\label{Proposition: bound reduction number} Let $\cA \subset \ZZ_{\geq 0}^{n}$ be a finite subset such that $\{0,v_1,\hdots,v_n\} \subset \cA$. 
 \begin{itemize}
     \item [(i)] If $\cH(\overline{\cA})$ contains all integral points of an $i$-dimensional face of $\conv(\overline{\cA})$, then 
     $r(\cA) \leq d_{\cA}^{n-i} + i - 1$. 
     \item [(ii)] If an $i$-dimensional face of $\conv(\overline{\cA})$ contains $q + i + 1$ points of $\cH(\overline{\cA})$, then 
     $r(\cA) \leq (d_{\cA}^{i}-q)d_{\cA}^{n-i}$. 
 \end{itemize}
 \end{proposition}
 \begin{proof}
 See \cite[Lemma 1.2 and 1.3]{Hoa-Stuckrad}. 
 \end{proof}
 
 \section{GT-sumsets and RL-varieties}
 \label{Section: GT-sumsets and RL-varieties.}

The aim of this section is to illustrate how the relationship between additive number theory and algebraic geometry allows us to go back and forth and solve interesting open problems. First, we introduce the notions of $GT-$subsets and $GT$-sumsets associated to linear systems of congruences. For this kind of subsets $\cA \subset \ZZ_{\geq 0}^{n}$ and using the geometry of $Y_{n,d_{\cA}}$, we provide a low bound for the phase transition $n_0(\cA)$ and families of examples for which the function $\varphi_{\cA}(t)$ and polynomial $p_{\cA}(t)$ are completely determined. Second, we present $RL-$varieties,  they are monomial projections of the Veronese variety $X_{n,d}$ intrinsically related to $GT-$subsets and $GT-$sumsets. Using properties of sumsets, we are able to compute their degree and their Hilbert polynomial.

\begin{notation}
Let $1 \leq n,d_1,\hdots,d_s$ be integers and let $M = (a_{i,j})$ be a $s \times 
(n+1)$ matrix of integers with $0 \leq a_{i,0}
,\hdots, a_{i,n} < d_i$ and $\GCD(a_{i,0}, \hdots, a_{i,n},d_i) = 1$, for each $1 \leq i \leq s$. We set $d := d_1 \cdots d_s$ and we denote by $(M,d_1,\hdots,d_s)$ the linear system of congruences
\begin{equation}\label{Equation: Integer system}
\left\{\begin{array}{llllcllcl}
y_0 & + & y_1 & + & \cdots & + &  y_n & \equiv & 0 \mod d\\
a_{1,0} y_0 & + & a_{1,1} y_1 & + & \cdots & + & a_{1,n}y_n & \equiv & 0 \mod d_1\\
    &   &                &   &  &   &            & \vdots  &  \\
a_{s,0}y_0 & + & a_{s,1} y_1 & + & \cdots & + & a_{s,n} y_n & \equiv & 0 \mod d_s.
\end{array}\right.
\end{equation}
For each $t \geq 1$, we denote by $(M,d_1,\hdots,d_s;t)$ the linear system:
\begin{equation}\label{Equation: Integer system graduated}
\left\{\begin{array}{llllcllcl}
y_0 & + & y_1 & + & \cdots & + &  y_n & = & td\\
a_{1,0} y_0 & + & a_{1,1} y_1 & + & \cdots & + & a_{1,n}y_n & \equiv & 0 \mod d_1\\
    &   &                &   &  &   &            & \vdots  &  \\
a_{s,0}y_0 & + & a_{s,1} y_1 & + & \cdots & + & a_{s,n} y_n & \equiv & 0 \mod d_s.
\end{array}\right.    
\end{equation}
\end{notation}

\begin{definition}\label{Definition: GT-subset} Let $(M,d_1,\hdots,d_s)$ be a linear system of congruences. The $\ZZ_{\geq 0}^{n+1}-$solutions of the system $(M,d_1,\hdots,d_s;1)$ form a finite subset $\overline{\cA} \subset \ZZ_{\geq 0}^{n+1}$ and we define the $GT-$subset associated to $(M,d_1,\hdots,d_s)$ to be $\cA = \{(a_1,\hdots, a_n) \mid (a_0,a_1,\hdots,a_n) \in \overline{\cA}\}\subset \ZZ_{\geq 0}^n$. 
\end{definition}

The set of all $\ZZ_{\geq 0}^{n+1}-$solutions of a system of congruences $(M,d_1,\hdots, d_s)$ is an affine semigroup $H \subset \ZZ_{\geq 0}^{n+1}$. By \cite[Theorem 2.2.11]{ThesisLiena} the associated subset $\overline{\cA}$ minimally generates the semigroup $H$. Thus, for each $t \geq 1$ the {\em $GT-$sumset} $t\cA$ is uniquely determined by the set of $\ZZ_{\geq 0}^{n+1}-$solutions of the system $(M,d_1,\hdots,d_s;t)$ and, hence, $$\varphi_{\cA}(t) = |t\cA| = |(M,d_1,\hdots, d_s,t)|.$$ 
Geometrically, the coordinate ring of the monomial projection $Y_{n,d_{\cA}}$ of the Veronese variety $X_{n,d}$ parameterized by the set $\Omega_{n,d}$ of monomials associated to the $GT-$subset $\cA$ of $(M,d_1,\hdots, d_s)$ is the semigroup ring $\kk[H]$. 

In addition, the above construction can be also interpreted from invariant theory point of view. Given a system of congruences $(M,d_1,\hdots,d_s)$, we set $e$ a $d$th primitive root of $1 \in \kk$. 
For each $1 \leq i \leq s$ we denote $e_i = e^{d/d_i}$ and by $M_{d_i;a_{i,0},\hdots, a_{i,n}}$ the diagonal $(n+1)\times(n+1)$ matrix $\diag(e_i^{a_{i,0}},\hdots, e_i^{a_{i,n}})$. Since by hypothesis $\GCD(a_{i,0}, \hdots, a_{i,n},d_i) = 1$, each matrix $M_{d_i;a_{i,0},\hdots, a_{i,n}}$ generates a cyclic subgroup $\Gamma_i$ of $\GL(n+1,\kk)$ of order $d_i$. We take $G = \Gamma_1 \oplus \cdots \oplus \Gamma_s \subset \GL(n+1,\kk)$, which is an abelian group of order $d$ acting diagonally on $R$ and we denote by $\bG \subset \GL(n+1,\kk)$ its cyclic extension, the abelian group generated by $G$ and the diagonal matrix $\diag(e,\hdots,e)$.  The ring $R^{\bG} = \{p \in R \mid g(p) = p, \, \forall g \in \bG\}$ of invariants of $\bG$ has a basis of monomials, precisely the monomials $x_0^{a_0}\cdots x_n^{a_n} \in R$ such that $(a_0,\hdots,a_n)$ is a $\ZZ_{\geq 0}^{n+1}-$solution of $(M,d_1,\hdots,d_s)$. By \cite[Theorem 2.2.11]{ThesisLiena}, $\Omega_{n,d}$ is a minimal set of generators of the ring $R^{\bG} = \kk[H]$.  So, for the phase transition $n_0(\cA)$ of a GT-subset we have:

\begin{proposition}\label{Proposition: Phase transition GT-subsets} Let $\cA$ be the $GT-$subset associated to a linear system of congruences  $(M,d_1,\hdots,d_s)$. Then, $n_0(\cA) \leq n+1$.
\end{proposition} 

\begin{proof}
By \cite[Theorem 3.3.5]{ThesisLiena}, it holds that $\reg (R^{\bG})+1=\reg(\cA)+1 \leq n+1$. Since $n_0(\cA) \leq \reg(\cA)+1$, the result follows.  
\end{proof}

The Castelnuovo--Mumford regularity $\reg(\cA)$ gives us a bound for the phase transition $n_0(\cA)$ for $GT-$sumsets which is considerable low compared to the bounds we have  in Section \ref{sec-phase}. Indeed, $R^{\bG} = K[H]$ is a Cohen-Macaulay ring, so we have by Theorem \ref{Theorem: Bounds reg Hoa-Stuckrad} that $n_0(\cA) \leq \deg(Y_{n,d}) - |\cA| + n + 2$. Now, by \cite[Proposition 3.1.2]{ThesisLiena} we have that $\deg(Y_{n,d}) \geq |\cA|-1$ and, hence, the claim $\deg(Y_{n,d}) - |\cA| + n+2 \geq n+1$ follows. Furthermore, the structure of $GT-$sumsets allows to 
completely determine the function $\varphi_{\cA}(t)$, the polynomial $p_{\cA}(t)$ and the phase transition $n_0(\cA)$ in many cases. The first approach consists of counting the number of solutions of the systems $(M,d_1,\hdots,d_s;t)$. This method depends on the system $(M,d_1,\hdots,d_s)$ and it is out of reach for large values of $n$. Nevertheless, for $n = 2,3$ it often leads to a complete solution. For instance, 

\begin{proposition}\label{Proposition: polynomial GT-surfaces} Let $2 \leq n < d$ be integers and $(M,d)$ the system of congruences:
\[\left\{\begin{array}{llllcllcl}
y_0 & + & y_1 & + & y_2 & \equiv & 0 \mod d\\
 &  & a_{1,1} y_1 & + & a_{1,2}y_2 & \equiv & 0 \mod d.\\
\end{array}\right.    
\]
We set $a' = \frac{a_{1,1}}{\GCD(a_{1,1},d)}$, $d' = \frac{d}{\GCD(a_{1,1},d)}$ and $0 < \lambda \leq d'$ the integer such that $a_{1,2} = \lambda a' + \mu d'$ with $\mu \in \ZZ$. 

Then, the phase transition $n_0(\cA)$ for the $GT-$subset $\cA \subset \ZZ_{\geq 0}^{2}$ of $(M,d)$ is zero and 
\[\varphi_{\cA}(t) = p_{\cA}(t) = \frac{dt^2 + \theta(a_{1,1},a_{1,2},d)t + 2}{2},\]
where $\theta(a_{1,1},a_{1,2},d) = \GCD(a_{1,1},d) + \GCD(\lambda,d') + \GCD(\lambda-\GCD(a_{1,1},d),d')$. 
\end{proposition}
\begin{proof}
The function $\varphi_{\cA}(t)$ coincides with the number of solutions of the system $(M,d;t)$ and the result follows from \cite[Theorem 4.12]{CMM-R}.
\end{proof}

The second approach is based on the computation of the  Hilbert series $\HS(A(Y_{n,d_{\cA}}),z) = \sum_{t \geq 0} \varphi_{\cA}(t)z^t$ of $Y_{n,d_{\cA}}$ and the fact that $A(Y_{n,d_{\cA}}) \cong R^{\bG}$. The Hilbert series can be obtained from the Molien series of $\bG$ as follows: 
\[\HS(A(Y_{n,d_{\cA}}),z^{d}) = \frac{1}{|\bG|} \sum_{g \in \bG} \frac{1}{\det(Id - zg)},\]
which is an expression that only depends on $(M,d_1,\hdots,d_s)$ (see, for instance, \cite{BH93}). The expansion of the Molien series of $\bG$ gives the function $\varphi_{\cA}(t)$ and the polynomial $p_{\cA}(t)$. For instance we have the following result, for sake of completeness we include a simple proof. 

\begin{proposition}\label{Proposition: polynomial d prime} Let $2 \leq n < d$ be integers with $d$ prime and $(M,d)$ a system of congruences:
\[\left\{\begin{array}{llllcllcl}
y_0 & + & y_1 & + & \cdots & + &  y_n & \equiv & 0 \mod d\\
 &  & a_{1,1} y_1 & + & \cdots & + & a_{1,n}y_n & \equiv & 0 \mod d\\
\end{array}\right.\]
with $0 < a_{1,1} < \cdots < a_{1,n}$. Then, the phase transition $n_0(\cA)$ for the $GT-$subset $\cA \subset \ZZ_{\geq 0}^{n}$ associated to $(M,d)$ is zero and 
\[\varphi_{\cA}(t) = p_{\cA}(t) = \frac{1}{d}\binom{td+n}{n} + \frac{d-1}{d}.\]
\end{proposition}

\begin{proof} For any $t \in \ZZ_{\geq 0}$, we have that $\varphi_{\cA}(t) = |(M,d_1,\hdots,d_s;t)|$ which is the number of monomial invariants of $\bG$ of degree $td$. Since it coincides with the number of monomial invariants of $G$ of degree $td$, it is enough to consider the expansion of the Molien series of $G$ in degree $td$: 
 
\[\frac{1}{d} \sum_{g \in G} \frac{1}{\det(Id-zg)} = \frac{1}{d} \sum_{k=0}^{d-1}\frac{1}{(1-z)(1-e^{ka_{1,1}}z)\cdots(1-e^{ka_{1,n}}z)}.\]
Since $d$ is prime and the exponents $0 < a_{1,1} < \cdots < a_{1,n} < d$, the classes of $ka_{1,1}, \hdots,ka_{1,n} \; \mod d$ are represented by two by two distinct integers in the set $\{0,\hdots,d-1\}$. Using the factorization $(1-z^d) = \prod_{j=0}^{d-1}(1-e^jz)$, we can write it as:
\[\frac{1}{(1-z^d)} \prod_{\substack{j \neq k a_{1,i} \mod d\\  i = 0,\hdots,n}} (1-e^jz), \]
which gives us the following expression:
\[
 \frac{1}{d} \left(\sum_{i=0}^{\infty} (-1)^{i}\binom{-(n+1)}{i}z^{i}
\!+\! \sum_{i=0}^{\infty} (\sum_{k=1}^{d-1} \prod_{\substack{j \neq k a_{1,i} \mod d\\  i = 0,\hdots,n}} (1-e^jz)) z^{id} \right).\]
 The expansion of the first summand at $z^{td}$ provides $\binom{td+n}{n}$. For each $1 \leq k \leq d-1$, $\displaystyle{\prod_{\substack{j \neq k a_{1,i} \mod d\\  i = 0,\hdots,n}} (1-e^jz)}$ is a polynomial in $z$ of degree strictly smaller than $d$, so the second provides $d-1$ at $z^{td}$ and the result follows.
\end{proof}

In the last part of this paper, we will use the results of sumsets obtained so far to derive new results about $RL-$varieties; $RL-$varieties were introduced and studied in \cite{CM-R1} and \cite{ThesisLiena}. They are a family of smooth rational monomial projections of the Veronese variety $X_{n,d}$ naturally related to $GT-$subsets. Basic information as the degree of an $RL-$variety was unknown and the approach and techniques we have presently developed will allow us to  obtain new information about this family of varieties and, in particular, to compute their degree and their Hilbert polynomial. 

Given a system of congruences $(M,d_1,\hdots,d_s)$, the set of $\ZZ_{\geq 0}^{n+1}-$solutions $(a_0,\hdots,a_n)$ satisfying $a_0\cdots a_n \neq 0$ is the {\em relative interior} $\relint(H)$ of the associated affine semigroup $H \subset \ZZ_{\geq 0}^{n+1}$. We define $\rl(\overline{\cA}) := \overline{\cA} \cap \relint(H)$, we denote by $\rl(\cA) \subset \ZZ_{\geq 0}^{n}$ the corresponding subset and by $\rl(\Omega_{n,d})$ its associated set of monomials. The $RL-$variety $\rl(Y_{n,d_{\cA}})$ associated to the $GT-$subset $\cA$ is the monomial projection of the Veronese variety $X_{n,d}$ induced by the subset $\rl(\cA)^{c}:= \{(a_1,\hdots,a_n) \in \ZZ_{\geq 0}^{n} \mid a_1+\cdots+a_n \leq d\} \setminus \rl(\cA)$, i.e. it is parameterized by $\cM_{n,d} \setminus \rl(\Omega_{n,d})$. 

\begin{proposition}\label{Proposition: degree RL-variety}
Let $n \geq 2$ be an integer and $(M,d_1,\hdots,d_s)$ a system of congruences with $GT-$subset $\cA \subset \ZZ_{\geq 0}^{n}$.
\begin{itemize}
    \item [(i)] The degree of the $RL-$variety $\rl(Y_{n,d_{\cA}})$ associated to $\cA$ is $d^n$.
    \item[(ii)] The Castelnuovo--Mumford regularity $\reg(\rl(\cA)^c) \leq \min\{n(d-2)+1, d^n-\binom{n+d}{n}+|\rl(\cA)|+n+2\}$.
    \item[(iii)] The phase transition $n_{0}(\rl(\cA)^{c}) \leq n+1$, moreover, for any $t \geq n+1$ we have
\[\varphi_{\rl(\cA)^c}(t) = p_{\rl(\cA)^c}(t) = \binom{td+n}{n}.\]    
\end{itemize}
\end{proposition}
\begin{proof} (i) The subset $\overline{\rl(\cA)^c}$ contains $(d,0,\hdots,0), \hdots, (0,\hdots,0,d)$ and any point $(0,\hdots, 1, \hdots, d-1, \hdots,0)$. Therefore, the convex hull of $\rl(\cA)^{c}$ is an $n-$simplex and  $\ZZ(\rl(\cA)^{c}-\rl(\cA)^{c}) = \ZZ^{n}$. Using (\ref{Equation: Khovanskii lead term}), it follows that $\deg(\rl(Y_{n,d_{\cA}})) = d^n$.

\noindent (ii) Since $\rl(Y_{n,d_{\cA}})$ is a smooth variety (\cite[Proposition 5.1.11]{ThesisLiena}), it follows from Theorem \ref{Theorem: Bounds reg Hoa-Stuckrad}.

\noindent (iii) \iffalse Observe that $\overline{\rl(\cA)^c}$ contains all points of all $n-$dimensional faces of the $n+1-$simplex with vertex $\{(d,\hdots,0), \hdots, (0,\hdots,d)\} \subset \ZZ_{\geq 0}^{n+1}$. Then, by Proposition \ref{Proposition: bound reduction number} we have that the reduction number $r(\rl(\cA)^c) \leq n$. Now the results follows from Theorem \ref{Theorem: Bounds reg Hoa-Stuckrad-reduction number}. \fi
The statement follows from the {\em claim}: {\em for any $t \geq n+1$, the $t$-fold sumset $t\overline{\rl(\cA)^c} = \{(a_0,\hdots,a_n) \in \ZZ_{\geq 0}^{n+1} \mid a_0+\cdots+ a_n = td\}$}. Indeed, let $t \geq n+1$ be an integer and $a = (a_0,\hdots,a_n) \in \ZZ^{n+1}_{\geq 0}$ be such that $a_0+\cdots +a_n = td$. We prove that $a \in t\overline{\rl(\cA)^c}$. Notice that if some $a_i = 0$, then it is immediate that $a \in t\overline{\rl(\cA)^c}$ since $\overline{\rl(\cA)^c}$ contains all points $\{(a_0,\hdots,a_n) \in \ZZ_{\geq 0}^{n+1} \mid a_0+\cdots+a_n = d \quad \text{and} \quad a_i = 0\}$. So, we can assume that $a_0 \geq 1,\hdots,a_n \geq 1$ and, without loss of generality, we consider $a_0 = \min\{a_0,\hdots,a_n\}$. Since $|a| = a_0+\cdots+a_n \geq (n+1)d$, there is $a_i$ such that $a_i \geq d$, otherwise $|a| < (n+1)d$ and we get a contradiction. We set $a_0 = t'd + k_0$ with $k_0 \in \{0,\hdots,d-1\}$ and we distinguish two cases. 

\noindent \underline{Case 1}: If $t' > 0$, then $a_1 > d$ and we can write 
\[a = (a_0, d-k_0,0,\hdots,0) + (0,a_1-d+k_0,a_2,\hdots,a_n) = a^1+a^2\]
with $a^1,a^2 \in \ZZ_{\geq 0}^{n+1}$. Then we have that $|a^1| = t'd+k_0+d-k_0 = (t'+1)d$ and $|a^2| = a_1-d+k_0 +a_2+\cdots +a_n = td-a_0-d+k_0 = td-t'd-k_0-d+k_0 = (t-t'-1)d$. Thus $a^1,a^2 \in \cH(\overline{\rl(\cA)^c})$ and, hence, $a \in t\overline{\rl(cA)^c}$. 

\noindent \underline{Case 2}: If $t' = 0$, then $a_0 = k_0 < d$ and, without loss of generality, we can assume that $a_1 \geq d$. Hence, we can write
\[a = (a_0, d-a_0,0,\hdots,0) + (0,a_1-d+a_0,a_2,\hdots,a_n) = a^1+a^2\]
with $a^1,a^2 \in \ZZ_{\geq 0}^{n+1}$. Arguing exactly as in Case 1, we obtain $a \in t\overline{\rl(cA)^c}$. \end{proof}

Actually, Proposition \ref{Proposition: degree RL-variety} is true for any monomial projection $Y_{n,d}$ of the Veronese variety $X_{n,d}$ in dimension $n \geq 2$ parameterized by a subset of monomials $\Omega_{n,d}$ obtained from the set of all monomials of degree $d$ in $R$ by deleting only monomials $x_0^{a_{0}} \cdots x_{n}^{a_n}$ with $a_0\cdots a_n \neq 0$, i.e. by deleting only monomials having all the variables $x_0,\hdots,x_n$.

We end this section with an illustrating example. 
\begin{example}\rm We take $n = 2$ and $\cA = \{(0,0), (3,0), (0,3), (1,1)\}$ with $d_{\cA} = 3$. We have that $\cA$ is the $GT-$subset of the linear system of congruences: 
\[\left\{\begin{array}{llllcllcl}
y_0 & + & y_1 & + & y_2 & \equiv & 0 \mod 3\\
 &  & y_1 & + & 2y_2 & \equiv & 0 \mod 3.\\
\end{array}\right.    
\]
It is straightforward to see that $\rl(\cA) = \{(1,1)\}$ and $\rl(\cA)^c = \{(0,0), (1,0), (0,1), (2,0), (0,2), (3,0),$ $ (2,1), (1,2), (0,3)\}$. We have $|\rl(\cA)^c| = 9$, $d_{\rl(\cA)^c} = 3$ and the associated simple  projection $Y_{2,3} \subset \PP^{8}$ of the Veronese surface $X_{2,3} \subset \PP^{9}$ parameterized by $\rl(\cA)^c$ is the $RL-$variety associated to the $GT-$sumset $\cA$ (Example \ref{Example: examples section 3}(ii)). 
By Proposition \ref{Proposition: degree RL-variety}, we obtain $\deg(Y_{2,3}) = 9$ as in Example \ref{Example: examples section 3}(ii). On the other hand, $\reg(\rl(\cA)^c) \leq 3$ and $n_0(\rl(\cA)^c) \leq 3$, which is very close to the real value $n_0(\rl(\cA)^c) = 2$. In addition, for all $t \geq n+1$ we have
\[\varphi_{\rl(\cA)^c}(t) = p_{\rl(\cA)^c}(t) = \binom{3t+2}{2}.\]

\end{example}

%%%%%%%%%%%%%%%%%%%%%%%%%%%%%%%%%%%%%%
%%%%%%%%%%%%%%%%%%%%%%%%%%%%%%%%%%%%%%%%%%%%%%%%%%

\end{document}